\documentclass[12pt]{amsart}
\usepackage{amssymb}
\usepackage{graphicx}
\usepackage{microtype,booktabs}
\usepackage[noadjust]{cite}
\usepackage{xspace}
\usepackage{longtable}
\usepackage{pdflscape}
\usepackage{mathtools}
\usepackage[pdfusetitle]{hyperref}
\hypersetup{pdfauthor={Steve Donnelly, Paul E. Gunnells, Ariah Klages-Mundt, Dan Yasaki}} 

\numberwithin{equation}{section}
\newtheorem{theorem}{Theorem}[section]

\newtheorem{prop}[theorem]{Proposition}

\newtheorem{heur}[theorem]{Heuristic}

\theoremstyle{definition}
\newtheorem{definition} [theorem]{Definition}
\newtheorem{example} [theorem] {Example}
\newtheorem{remark} [theorem] {Remark}

\newif\ifshowvc
\showvctrue
\showvcfalse  

\ifshowvc
\input{vc}
\fi

\newif\iflandscapetable
\landscapetabletrue

\DeclarePairedDelimiter{\set}{\{}{\}}
\newcommand{\sage}{\texttt{SAGE}\xspace}
\newcommand{\magma}{\texttt{Magma}\xspace}
\newcommand{\Q}{\mathbb{Q}}
\newcommand{\ZZ}{\mathbb{Z}}
\newcommand{\C}{\mathbb{C}}
\newcommand{\R}{\mathbb{R}}
\newcommand{\Z}{\mathbb{Z}}

\newcommand{\F}{\mathbb{F}}

\newcommand{\n}{\mathfrak{n}}
\newcommand{\fH}{\mathfrak{H}}

\newcommand{\PP}{\mathbb{P}}

\newcommand{\OF}{\mathcal{O}_F}
\newcommand{\OK}{\mathcal{O}_K}

\newcommand{\ap}{a_{\mathfrak{p}}}
\newcommand{\p}{\mathfrak{p}}
\newcommand{\N}{\mathfrak{N}}
\newcommand{\m}{\mathfrak{m}}
\renewcommand{\d}{\mathfrak{d}}

\newcommand{\q}{\mathfrak{q}}
\newcommand{\bG}{\mathbf{G}}
\newcommand{\D}{\mathcal{D}}

\DeclareMathOperator{\tors}{tors}

\DeclareMathOperator{\ord}{ord}
\DeclareMathOperator{\SL}{SL}
\DeclareMathOperator{\Orth}{O}
\DeclareMathOperator{\GL}{GL}
\DeclareMathOperator{\U}{U}
\DeclareMathOperator{\Gal}{Gal}
\DeclareMathOperator{\Res}{R}
\DeclareMathOperator{\res}{Res}
\DeclareMathOperator{\Red}{Red}%
\DeclareMathOperator{\End}{End}%
\DeclareMathOperator{\Norm}{N}
\def\BorelSerre{\text{BS}}
\def\Eis{\text{Eis}}

\renewcommand{\thefootnote}{\arabic{footnote}}
\begin{document}

\title[Table of elliptic curves]{A table of elliptic
curves over the cubic field of discriminant $-23$}

\author{Steve Donnelly}
\address{School of Mathematics and Statistics\\
University of Sydney\\
Sydney NSW 2006\\
AUSTRALIA }
\email{stephen.donnelly@sydney.edu.au}

\author{Paul E. Gunnells}
\address{Department of Mathematics and Statistics\\University of
Massachusetts\\Amherst, MA 01003\\
USA}
\email{gunnells@math.umass.edu}

\author{Ariah Klages-Mundt}
\address{Department of Mathematics\\ 
Amherst College\\Amherst, MA 01002\\
USA}
\email{aklagesmundt12@alumni.amherst.edu}

\author{Dan Yasaki}
\address{Department of Mathematics and Statistics\\ 
The University of North Carolina at Greensboro\\Greensboro, NC 27412\\
USA}
\email{d\_yasaki@uncg.edu}

\thanks{
PG wishes to thank the National Science Foundation for support of this
research through the NSF grant DMS-1101640.  AKM and PG both thank the
Amherst College Department of Mathematics for partial support.}

\renewcommand{\setminus}{\smallsetminus}

\date{\today}

\begin{abstract}
Let $F$ be the cubic field of discriminant $-23$ and $\OF$ its ring of
integers.  Let $\Gamma$ be the arithmetic group $\GL_2 (\OF)$, and for
any ideal $\n \subset \OF$ let $\Gamma_{0} (\n)$ be the congruence
subgroup of level $\n$.  In \cite{gunnells-yasaki12}, two of us (PG
and DY) computed the cohomology of various $\Gamma_{0} (\n)$, along
with the action of the Hecke operators.  The goal of
\cite{gunnells-yasaki12} was to test the modularity of elliptic curves
over $F$.  In the present paper, we complement and extend the results
of \cite{gunnells-yasaki12} in two ways.  First, we tabulate more
elliptic curves than were found in \cite{gunnells-yasaki12} by using
various heuristics (``old and new'' cohomology classes, dimensions of
Eisenstein subspaces) to predict the existence of elliptic curves of
various conductors, and then by using more sophisticated search
techniques (for instance, torsion subgroups, twisting, and
the Cremona--Lingham algorithm) to find them.  We then compute
further invariants of these curves, such as their rank and
representatives of all isogeny classes.  Our enumeration includes
conjecturally the first elliptic curves of ranks $1$ and $2$ over this
field, which occur at levels of norm $719$ and $9173$ respectively.
\end{abstract}

\maketitle
\ifshowvc
\let\thefootnote\relax
\footnotetext{Base revision~\GITAbrHash, \GITAuthorDate,
\GITAuthorName.}
\fi

\section{Introduction}\label{intro}

\subsection{}
Let $F$ be the cubic field of discriminant $-23$ and let $\OF$ be its
ring of integers.  Let $\bG$ be the reductive $\Q$-group $\Res_{F/\Q}
(\GL_{2})$, let $\Gamma\subset \bG (\Q)$ be the arithmetic group
$\GL_2 (\OF)$, and for any ideal $\n \subset \OF$ let $\Gamma_{0}
(\n)$ be the congruence subgroup of level $\n$.  In
\cite{gunnells-yasaki12} two of us (PG and DY) investigated the
modularity of elliptic curves over $F$.  In particular, for all ideals
$\n$ of norm up to some bound, we computed the action of the Hecke
operators on the cohomology of the congruence subgroup $\Gamma_{0}
(\n) \subset\GL_{2} (\OF)$ and identified classes with integral
eigenvalues that are apparently attached to cuspidal automorphic forms
on $\GL_{2}/F$.  For each such class $\xi$, we found an elliptic curve
$E/F$ of conductor $\n$ such that $a_{\p} (E) = a_{\p} (\xi)$ for all
primes $\p \nmid \n$ that we could check.  Here $a_\p (\xi)$ denotes
the eigenvalue of the Hecke operator $T_{\p}$ on $\xi$, and $a_{\p}
(E)$ comes from counting the points on $E$ over the residue field
$\F_{\p} = \OF /\p$:
\[
a_{\p} (E) = \Norm (\p)+1-\# E (\F_\p).
\]   

\subsection{}
In this paper, we complement and extend the results of
\cite{gunnells-yasaki12} in two ways.  First, we investigate more fully
the elliptic curves found in \cite{gunnells-yasaki12} by computing
invariants such as their torsion subgroups and Mordell--Weil ranks.
We also find representatives of the different isomorphism classes of
curves within an isogeny class.  

Second, we extend our table of curves through a variety of heuristics
inspired by results in \cite{gunnells-yasaki12}.  For instance, we use
a heuristic of ``old and new'' cohomology classes and observations
about the dimensions of Eisenstein subspaces in cohomology to make
predictions about the dimensions of cuspidal subspaces.  For many
levels this prediction gives a one-dimensional cuspidal space, which
then gives a prediction for the existence of an elliptic curve.  In
all such cases our searches yielded an apparently unique isogeny class
of elliptic curves over $F$ of that conductor.  For other levels our
heuristics predict cuspidal subspaces of dimension $>1$.  For some of
these levels we found multiple isogeny classes of curves; for others
we find no elliptic curves. We remark that most of these computations
involve levels whose norms are far beyond those of levels where Hecke
operator computations as in \cite{gunnells-yasaki12} are feasible.
Thus we have no way of checking the ``modularity'' these curves, or
even that the cohomology classes themselves appear to be attached to
Galois representations.  Nevertheless, in our opinion the fact that
cohomology predicts the existence of these curves merely through
dimension counts is compelling.\footnote{We note that recent
remarkable work of P.~Scholze \cite{scholze} explains how to attach
Galois representations to Hecke eigenclasses in the mod $p$ and
characteristic $0$ cohomology of certain locally symmetric spaces.  At
present the example we consider falls outside the scope of this work,
since our field $F$ is neither totally real nor $CM$.}

Our paper fits into the long tradition of elliptic curve enumeration,
the modern era of which began with Cremona's extensive tables of
curves over $\Q$ \cite{cremona97} and imaginary quadratic fields
\cite{cremona.hyptess}.  Cremona's work has inspired many other
efforts, including further work over $\Q$ \cite{stein-watkins, bmsw},
as well as enumeration over $\Q (\sqrt{5})$ \cite{bdkm+12} and $\Q
(e^{2\pi i/5})$ \cite{ghy12}.

\subsection{} We now give an overview of the contents of this paper.
In Section~\ref{section 2} we recall the setup from
\cite{gunnells-yasaki12} and explain how we computed cohomology.  We
also describe the main heuristics that allow us to extend our
computations far beyond that of \cite{gunnells-yasaki12}.  In
Section~\ref{section 3} we present various methods for attempting to
find an elliptic curve over $F$ of a given conductor.  In
Section~\ref{section 4}, we address how to find all curves that are
isogenous to a given curve $E$ defined over $F$ via an isogeny defined
over $F$. In Section~\ref{section 5} we state our results and give
tables that summarize various information about our dataset of
elliptic curves.  Finally, in Appendix \ref{appendix} we give a small
table of elliptic curves over $F$ of conductor norm $<1187$, along
with some of their most important invariants; we believe this table
gives a complete enumeration of isomorphism classes up to this bound.
The full dataset we computed, which includes curves with conductors of
norm up to approximately 20000 (with fairly complete data for curves
of norm conductor less that 11575), is available online through the
\emph{$L$-functions and modular forms database} (\texttt{lmfdb.org}
\cite{lmfdb}).

\subsection{Acknowledgments} We thank Avner Ash, William Casselman,
Haluk \c{S}eng\"{u}n, and Mark Watkins for their interest in this
work.  PG and DY thank the American Institute of Mathematics and NSF
for support.

\section{Cohomological automorphic forms and further heuristics} \label{section 2}

\subsection{}
Throughout this paper we write $F$ for the cubic field $\Q[x]/
(x^3-x^2+1)$ of discriminant $-23$.  We let $a$ be a fixed root of
$x^3-x^2+1$.  The ring of integers $\OF$ is then $\Z [a]$, and the
unit group is generated by $-1$ and $a$.

In this section, we recall the setup of \cite{gunnells-yasaki12}.  As
above $\Gamma_{0} (\n)$ is a congruence subgroup of $\Gamma = \GL_{2}
(\OF)$.  Instead of trying to work directly with automorphic forms on
$\bG$, we compute the cohomology of $\Gamma_{0} (\n)$; by a theorem of
Franke \cite{franke} this allows us to work with certain automorphic
forms over $F$, including those that should be attached to elliptic
curves.   Let $C$ be the positivity domain of
positive-definite binary quadratic forms over $F$, as constructed by
Koecher (cf.\cite[\S 3]{gunnells-yasaki12} and \cite[\S 9]{koecher}).
The group $\Gamma$ acts on $C$, and induces an action on $C$ mod
homotheties, which can be identified with the global symmetric space
for $G = \bG (\R) \simeq \GL_{2} (\R)\times \GL_{2} (\C)$.  More
precisely, let $K \simeq \Orth(2) \times \U (2)$ be a maximal compact
subgroup of $G$ and let $A_{G}$ be the split component.  Then we have
an isomorphism
\begin{equation}\label{eq:symspace}
C/\R_{>0} \simeq G/A_{G}K \simeq \fH \times \fH_3 \times \R,
\end{equation}
where $\fH$ (respectively, $\fH_3$) is the hyperbolic plane (resp.,
hyperbolic $3$-space).  The explicit reduction theory due to Koecher
enables us to construct a $\Gamma$-equivariant decomposition of $C$
into polyhedral cones that induces a $\Gamma$-equivariant
decomposition of $C/\R_{>0}$ into cells.  The homology of the
associated chain complex over $\C $ mod $\Gamma_{0} (\n)$ can be
identified with $H^{*} (\Gamma_{0} (\n); \tilde{\Omega}_{\C} )$; here
$\tilde{\Omega}_{\C} $ is the system of local coefficients attached to
$\Omega \otimes \C$, where $\Omega$ is \emph{orientation module} of
$\Gamma$.\footnote{We take this time to point out an error in
\cite{gunnells-yasaki12}, in which we neglected to include the
orientation module in our coefficients.  None of the results there or
here are affected by this oversight.}

\subsection{} Over $F$, we have two sets of cohomological data on
automorphic forms.  First, we have computed the cohomology spaces
$H^{4} (\Gamma_{0} (\n); \tilde{\Omega}_{\C} )$ and Hecke operators on
levels up to norm $911$; we then expect the cuspidal eigenclasses with
rational eigenvalues to correspond to elliptic curves over
$F$. Second, for many levels of norm higher than $911$, including all
of the levels of norm less than or equal to $11569$, we have computed
the spaces $H^{4} (\Gamma_{0} (\n); \tilde{\Omega}_{\C} )$ but
\emph{no} Hecke operators.\footnote{The Hecke computations became
impractical at these levels because of our implementation.  With
better code we could undoubtedly treat some levels above norm $911$,
but even with this we do not expect to handle level norms above
$5000$.}  This means we cannot predict with certainly which ideals
should be conductors of elliptic curves.

Nevertheless, all is not lost.  To extend our computations beyond
level norm 911, we apply two heuristics derived from examination of
the Hecke data where we can compute Hecke operators.  The first
concerns the size of the Eisenstein subspace of the cohomology, and
the second concerns lifts of cohomology classes from lower levels to
higher.

\subsection{}
First, the Eisenstein cohomology is the cohomology that ``comes from
the boundary,'' and that should be eliminated from consideration when
one wants to predict the size of the cuspidal subspace.  For more
details about Eisenstein cohomology, we refer to \cite{harder.icm};
here we only recall the definition.  Let $X = G/A_{G}K$ be the global
symmetric space \eqref{eq:symspace}, and let $X^{\BorelSerre}$ be the
partial compactification constructed by Borel and Serre \cite{bs}.
The quotient $Y := \Gamma _{0} (\n )\backslash X$ is an orbifold, and
the quotient $Y^{\BorelSerre } := \Gamma_{0} (\n )\backslash
X^{\BorelSerre}$ is a compact orbifold with corners.  The local system
can be extended to the boundary, and we have
\[
H^{*} (\Gamma_{0} (\n); \tilde{\Omega}_{\C}) \simeq H^{*} (Y;
\tilde{\Omega}_{\C}) \simeq H^{*} (Y^{\BorelSerre};
\tilde{\Omega}_{\C}),
\]
where we have abused notation by denoting the original local system
and its extension by the same symbol.

Now let $\partial Y^{\BorelSerre} = Y^{\BorelSerre}\smallsetminus Y$.
The Hecke operators act on the cohomology of the boundary $H^{*}
(\partial Y^{\BorelSerre}; \tilde{\Omega}_{\C})$, and the inclusion of
the boundary $\iota \colon \partial Y^{\BorelSerre} \rightarrow
Y^{\BorelSerre}$ induces a map on cohomology $\iota^{*}\colon H^{*}
(Y^{\BorelSerre}; \tilde{\Omega}_{\C}) \rightarrow H^{*} (\partial
Y^{\BorelSerre}; \tilde{\Omega}_{\C})$ compatible with the Hecke
action.  The kernel $H^{*}_{!}  (Y^{\BorelSerre};
\tilde{\Omega}_{\C})$ of $\iota^{*}$ is called the \emph{interior
cohomology}; it equals the image of the cohomology with compact
supports.  The goal of Eisenstein cohomology is to use Eisenstein
series and cohomology classes on the boundary to construct a
Hecke-equivariant section $s\colon H^{*} (\partial Y^{\BorelSerre};
\C) \rightarrow H^{*} (Y^{\BorelSerre}; \C)$ mapping onto a complement
$H^{*}_{\Eis} (Y^{\BorelSerre}; \C )$ of the interior cohomology in
the full cohomology.  The image of $s$ is called the \emph{Eisenstein
cohomology}.  Computations from \cite{gunnells-yasaki12} suggest the
following:

\begin{heur}\label{conj:eis}
The Eisenstein subspace of $H^4(\Gamma_0(\n);\tilde{\Omega}_{\C})$ is
rank $2c(\n) - 1$, where $c(\n)$ is the number of $\Gamma_0(n)$-orbits
on $\PP^1(F)$.
\end{heur}

We remark that in principle we should be able to apply results of
Harder \cite{harder.gl2} to explicitly determine this subspace.
However, in practice it is easier to compute the Hecke operators on
$H^{4}$ and to determine how large the space is from the Hecke
eigenvalues (one looks for classes on which $T_{\p}$ acts with
eigenvalue $\Norm (\p) + 1$.)

\subsection{}
The second heuristic concerns how cuspidal eigenclasses at one level can
appear at another.  The data suggests that some of the same
considerations in the Atkin--Lehner theory of oldforms \cite{al} apply
in cohomology.  Recall that this theory is based on the observation
that if $f (z)$ is a holomorphic weight $k$ cuspform on $\Gamma_{0} (m)
\subset \SL_{2} (\Z)$, then $f (dz)$ is a holomorphic weight $k$
cuspform on $\Gamma_{0} (m')$ for any $m'$ divisible by $m$, where $d$
is any divisor of $m'/m$.  We observe the same in cohomology, which
leads to the following prediction:

\begin{heur}\label{conj:oldform}
Let $\xi$ be a cuspidal Hecke eigenclass at level $\n \subset \OF$,
and let $\N \subset \OF$ be divisible by $\n$. Then for every proper,
nontrivial divisor $\d$ of $\N/\n$, there is a cuspidal Hecke
eigenclass $\xi_{\d}$ in the cohomology at level $\N$ whose
eigenvalues agree with those of $\xi$ for $T_{\p}$ with $\p \nmid \N$.
Moreover, the classes $\xi_{\d}$ are linearly independent in
cohomology.
\end{heur}

We remark that this heuristic should follow from Casselman's
generalization of Atkin--Lehner theory to automorphic representations
of $\GL_{2}$ \cite{casselman-restriction, casselman-atkin-lehner}.
However, we have not checked the details of this computation.

\begin{example}
Let $\p_{5}$, $\p_{7}$, and $\p_{37}$ denote the degree 1 primes above
$5$, $7$, and $37$, respectively, and let $\N = \p_5 \p_{7} \p_{37}$.
The cohomology $H^4(\Gamma_0(\N);\tilde{\Omega}_{\C})$ is
$19$-dimensional.  Since $F$ has class number one, \cite[Theorem
7]{cremona-aranes} implies that the number of boundary components in
the Borel--Serre is 
\[c(\N) = \sum_{\d \mid \N} \phi_u(\d + \N \d^{-1}),\]
where 
\[ \phi_u(\m) = \#((\OF/\m)^\times/\OF^\times).\]
We compute that $\phi_u(\d + \N \d^{-1}) = 1$ for each of the $8$
divisors of $\N$, and so $c(\N) = 8$.  Thus the expected cuspidal
cohomology in $H^4(\Gamma_0(\N);\tilde{\Omega}_{\C})$ is
$4$-dimensional.  At level $\n = \p_5\p_7$ we find a $1$-dimensional
cuspidal cohomology space and an elliptic curve of conductor $\n$ to
account for it.  Since $\N / \n = \p_{37}$ has two proper nontrivial
divisors, Heuristic~\ref{conj:oldform} tells us that we should expect
a 2-dimensional contribution to the cohomology at level $\N$.
Similarly, the same happens at level $\n' = \p_5\p_{37}$ which again
produces a 2-dimensional contribution to the cohomology at level $\N$.
Therefore we expect (i) all the cuspidal eigenclasses at level $\N$
are accounted for by cohomology for the levels $\n , \n'$, and (ii) no
other levels dividing $\N$ should have cuspidal cohomology.  We find
that this is true, and thus do not expect to find any elliptic curves
over $F$ of conductor $\N$.  Indeed, applying the techniques in
Section~\ref{section 3} produced no curves over $F$ of this conductor.
\end{example}

\section{Strategies to find an elliptic curve} \label{section 3}

\subsection{}
In this section, we describe various strategies for finding an
elliptic curve $E$ over $F$; some of these are described in
\cite{bdkm+12} (for $F=\Q (\sqrt{5})$).  There are different
strategies to employ, depending on how much information one has about
$E$.  At the very least, one begins with an ideal $\n \subset \OF$
that one hopes is the conductor of an elliptic curve.  If one is
lucky, one has a list of the Hecke eigenvalues $a_{\p}$ for a range of
primes $\p$ that are supposed to match the point counts of $E (\OF
/\p)$; such data opens the door to other techniques.  However, it
should be emphasized that, unlike the case of elliptic curves over
$\Q$, even if one has complete explicit information about the
automorphic form $f$ on $\GL_{2}/F$ conjecturally attached to $E$,
there is no direct way to construct an elliptic curve $E_{f}$ with
matching $L$-function.  In other words, there is no known way to
produce the period lattice $\Lambda \subset \C$ such that $E_{f}
\simeq \C /\Lambda$.  (For a discussion of these issues over real
quadratic fields see \cite{lassina}).

\subsection{Naive Enumeration}

The most naive strategy is to systematically loop through Weierstrass
equations
\begin{equation}\label{eq:weier}
E \colon y^2 + a_1 xy + a_3 = x^3 + a_2 x^{2} + a_4 x + a_6, 
\end{equation}
with $a_1, a_2, a_3, a_4, a_6 \in \OF$ contained in some bounded
subset of $\OF $.  For each elliptic curve $E$, we can compute the
conductor $\n_{E}$ to see if it matches the prediction from
cohomology.  If we have Hecke data, we can then check if it seems to
agree with $E$.

This describes an algorithm that in principle will find all elliptic
curves over $F$; however, it is of course of no use as soon as the
curve with smallest Weierstrass coefficients in the target isogeny
class has large coefficients in any equation. For example, enumerating
all integral Weierstrass equations with two-digit coefficients over a
cubic number field requires on the order of $200^{18}$ computations,
which is infeasible. Most of the curves in our dataset could not be
found with this technique.  If one knows some $a_{\p}$s, then gains
can be made by sieving equations using congruence conditions imposed
on the coefficients; still this is too inefficient to find curves with
large Weierstrass coefficients.

\subsection{Torsion families}

We can refine the naive search in some cases if we can guess the
torsion subgroup structure of $E_f$. If the torsion subgroup of $E_f$
is one of the groups mentioned in Mazur's theorem or contains such a
subgroup, we can use the parametrizations of \cite{kubert76} to
significantly reduce our search area.

We use the following proposition to determine in which family to
search:
\begin{prop}\label{prop:1}
Let $\ell$ be a prime in $\Z$, and $E/F$ an elliptic curve. Then $\ell
\mid \#E'(F)_{\tors}$ for some curve $E'$ in the $F$-isogeny class of
$E$ if and only if for all odd primes $\p$ at which $E$ has good
reduction $\ell \mid \Norm(\p) + 1 - \ap$.
\end{prop}

\begin{proof}
One direction is easy.  Suppose $\ell \mid \#E'(F)_{\tors}$. Then by
the injectivity of the reduction map at primes of good reduction,
$\ell \mid \#{E}'(\OF/\p) = \Norm(\p) + 1 - \ap$. For the more
difficult converse, see \cite{katz81}.
\end{proof}

We can determine whether a curve in the isogeny class of $E_f$ likely
contains a $F$-rational $\ell$-torsion point by applying Proposition
\ref{prop:1} for all $\ap$ up to some bound on $\p$. If this is the
case, then we can search over the families of curves with
$\ell$-torsion for a curve in the isogeny class of $E_f$. Within a
relatively small search space, we can find many curves with large
coefficients much more quickly than with the naive search. For
example, we found the curve
\begin{multline*}
y^2 + a^2xy + a^2y 
= x^3 + (a+1)x^2 + (-200a^2+56a+5)x - 739a^2 + 41a + 1139
\end{multline*}
with conductor $(a^2 - 9)$ of norm $665$ and the curve
\begin{multline*}
y^2 + (a^2+1)xy + ay \\
= x^3 + (-a^2+a+1)x^2+ (-249910a^2+438560a-331055)x \\
+ 86253321a^2-151364024a+114261323
\end{multline*}
with conductor $(3a^2 -14a +1)$ of norm $2065$ by searching for curves
with $F$-rational $6$-torsion.\footnote{The given equation of the
second curve is the canonical model, which is a global minimal
model. The curve actually found using this method had the
coefficients $[a_{1},a_{2},a_{3},a_{4},a_{6}]=[16a^2 + 24a + 10, -1872a^2 -
152a + 952, -1872a^2 - 152a + 952, 0, 0]$.}

\subsection{Twisting}

Recall that \emph{quadratic twist} $E'/F$ of an elliptic curve $E/F$
is a curve that is isomorphic to $E$ over a degree 2 extension of
$F$. If we know an elliptic curve $E/F$ of some conductor, we can
compute quadratic twists to generate more curves over $F$,  and under
favorable conditions have information about the conductors of the
twists.  To make this precise, suppose the $j$-invariant $j (E)$ does
not equal $0, 1728$.  If $E$ has Weierstrass equation 
\[
E\colon y^{2} = x^{3} + \alpha x + \beta , \quad \alpha ,\beta \in F,
\]
then for $d\in \OF$ we define the $d$-twist $E^{d}$ by 
\begin{equation}
E^{d} \colon dy^2 = x^3 + ax + b
\end{equation}
We have the following well known proposition (for a proof, see
\cite{bdkm+12}):

\begin{prop}\label{prop:ariah.prop}
Let $E/F$ be an elliptic curve with $j \neq 0,1728$. If $\n$ is the
conductor of $E$ and the ideal generated by $d \in \OF$ is non-zero,
square-free, and coprime to $\n$, then the conductor of $E^{d}$ is
divisible by $d^2\n$.
\end{prop}

Given $E/F$, we can use Proposition~\ref{prop:ariah.prop} to find the
finite set of all $d$ such that $E^d$ might have norm conductor less
than a given bound. We can then compute the quadratic twists by these
$d$ to find curves that may otherwise be difficult to find. For
example, we found the curve 
\begin{multline*}
y^2 + (a+1)xy + (a^2+a+1)y \\
= x^3 + (-a^2-a)x^2 + (-43a^2+63a-69)x -
198a^2 + 335a - 288
\end{multline*}
with conductor $(14a - 3)$ and norm conductor $2645$ using this
method.  This curve is a quadratic twist of $y^2 + axy + ay = x^3 +
(a+1)x^2 + (6a-5)x + 4a^2-7a+2$, which was found by searching over
torsion families.  Another example is
\begin{multline*}
y^2 + (a^2+a)xy + a^2y \\
= x^3 + (-a^2-a)x^2 + (-212a^2+305a-181)x -
1422a^2 + 2466a - 2087
\end{multline*}
with conductor $(-15a^2 + 8a - 1)$ and norm conductor $3025$.  This is
a quadratic twist of $y^2 + (4a^2+3a+1)xy + (4a^2+3a)y = x^3
+ (4a^2+3a)x^2$, which was again found by searching over torsion
families.


\subsection{{Curves with prescribed good reduction}}
\label{cremona-lingham}

We also employ an algorithm due to Cremona--Lingham
\cite{cremona-lingham07}, which finds all elliptic curves with good
reduction at primes outside of a finite set $\mathcal{S}$ of primes in
a number field. A \magma \cite{magma} implementation of this algorithm
was provided by Cremona.  The algorithm has the advantage that it
allows targeting a specific conductor.  The drawback is that it can be
difficult to use in practice, since a key step involves finding
$\mathcal{S}$-integral points on elliptic curves.

\begin{definition}
The \emph{$m$-Selmer groups} $F(\mathcal{S},m)$ of $F^*$ are defined to be
\begin{equation*}
F(\mathcal{S},m) = \{x \in F^*/(F^*)^m \mid \ord_\p(x) \equiv 0
\bmod{m} \ \text{for all} \ \p \notin \mathcal{S}\}, 
\end{equation*}
where $F^*$ is the multiplicative group of $F$.
\end{definition}

\begin{definition}
$F(\mathcal{S},m)_{mn}$ is defined to be the image of the natural map
\[F(\mathcal{S},mn) \rightarrow F(\mathcal{S},m).\]
\end{definition}

The Cremona--Lingham algorithm computes the finite $m$-Selmer groups
$F(\mathcal{S},m)$ of $F^*$ for $m = 2$, $3$, $4$, $6$, and $12$.
>From these it computes a finite set of possible $j$-invariants such
that each elliptic curve with good reduction outside $\mathcal{S}$ has
$j$-invariant in this set. These $j$-invariants are either $j=0$ or
$1728$, cases which can be treated directly, or $\mathcal{S}$-integers
in $F$ satisfying
\begin{equation*}
w \equiv j^2(j-1728)^3 \bmod{F^{*6}} \quad \text{for $w \in F(\mathcal{S},6)_{12}$.} 
\end{equation*}
In the latter case $j$ is of the form $j = x^3/w = 1728 + y^2/w$,
where $(x,y)$ is a $\mathcal{S}$-integral point on the elliptic curve
$E_w: Y^2 = X^3 - 1728w$, of which there are finitely many by Siegel's
Theorem. From this set of $j$-invariants, we construct each curve with
the desired reduction properties (indeed, there are finitely many by
Shafarevich's Theorem): for each $j = x^3/w$ (excluding $j = 0, 1728$,
which are treated separately), we choose $u_0 \in F^*$ such that
$(3u_0)^6w \in F(\mathcal{S},12)$, and each curve is either of the
form $E: Y^2 = X^3 - 3xu_0^2X - 2yu_0^3$ or is a quadratic twist
$E^{(u)}$ for some $u \in F(\mathcal{S},2)$. We must also check that
each curve found has good reduction at the primes above 2 and 3 (if
these primes are not in $\mathcal{S}$).

The advantage of this approach is that it not only gives a way to
find curves of given conductor, but also to prove there are no others.  
The disadvantage is that it is usually feasible to carry this through
only for the smallest fields and conductors; the conductors in this paper
are already too big.  This is because, first of all, a large number of curves
$E_w$ must be considered individually.  Worse still, for many $w$ it is too 
hard to determine the set of all $\mathcal{S}$-integral points on $E_w$ over $F$.
The general method currently used for this involves first determining all
{\it rational} points, i.e.~determining the Mordell--Weil group $E_w(F)$.
This is inherently very difficult.  In particular, many of the groups $E_w(F)$ 
have rank $1$ and are generated by a point of huge height (as predicted by the
conjecture of Birch and Swinnerton-Dyer), and these generators are impossible
to find with current techniques for curves over number fields.
Ironically, in these hard cases there are never any $\mathcal{S}$-integral
points in $E_w(F)$, because those points won't have huge height.  So these
hard cases are of no interest to us, but we can't prove it without knowing
the generators!  

Despite these difficulties, we used the
Cremona--Lingham algorithm to find many curves with large coefficients,
curves that would have been virtually impossible to find by the previous methods.
Our implementations do not attempt to find all rational or $\mathcal{S}$-integral
points but simply search, in natural search regions, for points in $E_w(F)$.
For example, a search on $E_w$ found the following curve defined
over $F$ with $\n = (a^2 - 10a + 1)$ and norm conductor $865$ using
this method, which lies just outside the range of curves found in \cite{gunnells-yasaki12}:
\begin{equation*}
y^2 + axy + (a^2+1)y = x^3 + (-a^2-1)x^2 + (-48a^2+85a-63)x - 211a^2 + 368a - 277.
\end{equation*}

\subsection{A well-optimized search algorithm}
\label{elliptic-curve-search}

This section describes a more sophisticated algorithmic approach to using
the ideas of the Cremona--Lingham method of the previous section.  Again,
we abandon the goal of proving completeness: our primary goal is to find all
curves that actually exist.  (Naturally one also wishes to prove non-existence
of other curves, but this is simply too hard a problem with current algorithms.)
Having adopted this attitude, in dealing with the large number of candidate
curves $E_w$ we are free to focus effort wherever we choose, and to switch
between the candidates at will.
Furthermore, we bring to bear some powerful techniques for searching for
points on candidate curves.  We have a two-pronged approach: the two main
techniques described below complement each other to some extent (a point that
is hard to find for one of them is not necessarily so hard for the other).

The program that performs all this is carefully written so as to
minimize the effort required, starting with very quick searches on
each candidate and gradually increasing the effort.  It balances the
running times of the different techniques, and focuses more effort on
``more likely'' candidates according to some theoretical heuristics.
This program is implemented for general number fields, and is included
in the Magma computational algebra system: the function is called
{\tt EllipticCurveSearch}.

The first main technique is a direct search for points on $E_w$ which
targets points especially likely to be of interest.  This is based on
a heuristic idea due to Elkies:
if an elliptic curve has discriminant $d$ and invariants $c_4, c_6$,
then it is likely that for each archimedian absolute value $v$,
$|c_4^3|_v, |c_6^2|_v$ and $|1728d|_v$ are all of roughly the same size.
(If not, then there is a lot of cancellation in $c_4^3 - c_6^2 = 1728d$,
and one expects this to occur not so frequently).  Therefore we search for
points on $E_w : y^2 = x^3 - 1728d$ by running over small values for $x$
under a weighted norm that is determined by $d$.  We also put in some
non-archimedian information about $c_4$, so the search spaces consist
of all $x \in F$ in the intersection of some $\ZZ $-module with some ``box.''

The second main technique is a tuned version of the generic approach to
determining generators for the Mordell--Weil group of an elliptic curve
over a number field, using the method of two-descent.
Two-descent helps in two ways.  First of all, it gives an upper bound
on the rank of the Mordell--Weil group.  In particular, when the bound
is zero, or equals the rank of the group generated by points already
found, we are done with $E_w$.  Two-descent also gives a finite set of
``two-covering curves'' $C$ with covering maps to $C \to E_w$, such that
every point in $E_w(F)$ is the image of an $F$-rational point on some $C$.
The advantage is that such a point has smaller height than its image on
$E_w$, {\it if} one uses ``nice'' (i.e.~minimized and reduced) models of
the two-coverings.  An algorithm for minimizing and reducing two-coverings
over number fields is due to one of us (SD) and Fisher.  Additionally, many
two-coverings that have no rational points can be ruled out by computing
Cassels--Tate pairings; an algorithm for this is due to one of us (SD).

All the above-mentioned algorithms have good implementations in \texttt{Magma}, so 
are available for use in our search for elliptic curves of given conductor.
We explain how the search program works by describing what happens for some
particular levels.

For level $\n  = (9a^{2}-a-15)$ of norm $2879$, the space of forms has
dimension $2$, and there are two isogeny classes of elliptic curves.
The search program has to individually consider 144 candidate curves
$E_w$.  We give details about the two values of $w$ which yield the
two curves.

For $w = a^2 - 24a - 17$, $E_w$ has Mordell--Weil rank $3$.  Quick searches on $E_w$ find two
independent points; integral points in this rank $2$ subgroup yield three elliptic curves,
but none of conductor $N$.  Using two-descent, a third independent point is quickly found
(on the first two-cover chosen).  Integral points in the full rank $3$ group yield three
more elliptic curves, including the curve with conductor $N$ and discriminant $w$.

For $w = 17a^2 - 16a - 24$, $E_w$ has Mordell--Weil rank 2.  Quick searches on $E_w$
find no rational points.  Two-descent proves (first of all) that $\operatorname{rank} E_w(F) \le 2$.
In such cases, it is less likely that $E_w$ has rank $2$, than that it has rank $0$ and
that the two-coverings have no rational points, and indeed that a stronger condition holds,
namely that Cassels--Tate pairings between distinct two-coverings are nontrivial.
Therefore the program calculates the pairing, which turns out to be trivial in this case.
Next, the program searches on reduced models of (two of the) two-coverings,
obtaining two independent generators of $E_w(F)$.  An $\mathcal{S}$-integral point in
the group yields the second elliptic curve of conductor $\n $ (and discriminant $a^6w$).

The program spent a few seconds for each of these discriminants, mostly spent reducing
the two-coverings.  The entire process of finding the two curves of conductor $\n $ took
a minute or so.  This involves some luck, in that the ``right'' values of $d$ were among
the first few discriminants for which the program chose to apply the harder techniques
(two-descent etc).  Some heuristics are used in this guesswork, aiming to test the
more likely discriminants first, so it is a game of both strategy and luck.

For level $(-9a^{2}-11a+3)$ of norm 2915, the space of forms has dimension $3$ and there are
three isogeny classes of elliptic curves.  These were all found without using two-descent.
There were 5184 candidate discriminants; the entire process took about five minutes.
The curves found came from $E_w$ of rank 3, 1 and 2 (in order of search effort required).

On the other hand, for many levels the space is not entirely composed of elliptic curves,
and we do not have a good way to predict whether there should be elliptic curves.  For such
levels we must run the program, with some chosen setting of the ``overall effort'' parameter,
on the full set of candidates $E_w$.  A typical such level is $(12a^{2}+7a+4)$ of norm 3325,
where the space has dimension $3$ and there is (apparently) only one isogeny class.
It took several hours to process all 5184 candidate discriminants using all the techniques.


\section{Enumerating the curves in an isogeny class} \label{section 4}

\subsection{} Now we turn to the next step in our table-building:
given an elliptic curve $E/F$, we find representatives of all
isomorphism classes of elliptic curves $E'/F$ that are isogenous to
$E$ via an isogeny defined over $F$. Recall that two elliptic curves
in an isogeny class are linked by a chain of prime degree isogenies;
in particular, to enumerate an isogeny class we need to find all
isogenies of prime degree, of which there are finitely many for curves
that do not admit CM over the given number field. Over $\Q$, there is
an algorithmic solution to this problem based on the following (see
\cite{cremona97}):

\begin{enumerate}
\item Mazur's theorem, which that states that if $\psi\colon E \rightarrow
E'$ is a $\Q$-rational isogeny of prime degree, then $\deg \psi \leq
19$ or is in $\{37, 43, 67, 163 \}$ \cite{mazur}.
\item V{\'e}lu's formulas, which provide an explicit way to enumerate
all prime degree isogenies with a given domain $E$ (see \cite[III
Prop.~4.12]{silverman92} or \cite[III Section 3.8]{cremona97}).
\end{enumerate}

\subsection{}
V\'elu's formulas are valid for any number field and are implemented
in \sage and \magma, but there is currently no generalization of Mazur's
theorem that gives us an explicit bound on the possible prime degree
isogenies defined over a general number field.

Since we are interested in specific isogeny classes, we solve this
problem by taking a less general perspective: we determine which prime
degree isogenies are possible for a specific isogeny class using the
following well-known result:

\begin{theorem} \label{Galois reps}
Let $E$ be an elliptic curve over a number field $K$. For each prime number $\ell \in \Z$, let
\begin{equation*}
\rho_{E,\ell}\colon \Gal(\overline{\Q}/K) \rightarrow \GL(E[\ell])
\cong \GL_2(\Z/\ell\Z) \end{equation*} be the associated Galois
representation on $\ell$-torsion points, where $E[\ell]$ is the set
(actually group) of $\ell$-torsion points in $E(\overline{K})$. There
exists an isogeny $E \rightarrow E'$ defined over $K$ of prime degree
$\ell$ if and only if $\rho_{E,\ell}$ is reducible over
$\mathbb{F}_\ell$. In particular, if $\rho_{E,\ell}$ is irreducible
(over the algebraic closure of $\mathbb{F}_\ell$), then there can be
no isogenies $E \rightarrow E'$ of prime degree $\ell$.
\end{theorem}

In what follows, we describe our implementation of an algorithm due to
Billerey \cite{billerey:isogenies} that outputs a provably finite list
of primes $p$ such that a given elliptic curve $E$ over a number field
$K$ might have a $p$-isogeny.  We first develop the necessary
background in Section \ref{sec:isogenies1}, and then describe the
implementation of algorithm in Section \ref{sec:isogenies2}.

\subsection{} \label{sec:isogenies1}
Let $M \subset \ZZ[X]$ be the subset of all monic polynomials that do
not vanish at $0$.  For $P,Q \in M$, define $P*Q \in M$ by 
\begin{equation}
  (P*Q)(X) = \res_Z(P(Z),Q(X/Z)Z^{\deg Q}),
\end{equation}
where $\res_Z$ is the resultant with respect to $Z$.  This defines a
commutative monoid structure on $M$ with neutral element $\psi_1(X) =
X-1$ \cite[Lemma~2.1]{billerey:isogenies}.  For $r \geq 1$ and $P \in
M$, define $P^{(r)} \in M$ by  
\begin{equation}
P^{(r)}(X^r) = (P*\Psi_r)(X), \quad \text{where $\Psi_r(X) = X^r -1$.} 
\end{equation}

Let $K$ be a number field of odd degree $d$, and fix an elliptic curve
$E/K$ that does not admit CM over $K$. Let $\ell \in \ZZ$ be a prime
number such that $E$ has good reduction at every prime ideal of $\OK$
dividing $\ell \OK$. By abuse of language, we say that $E$ has good
reduction at $\ell$. In this case, let  
\begin{equation*}
\ell \OK = \prod_{\q_i \mid \ell} \q_i^{v_{\q_i}(\ell)} 
\end{equation*}
be the prime factorization of $\ell \OK$. Associate to $\ell$ the
polynomial 
\begin{equation*}
P_\ell^* = P_{\q_1}^{(12v_{\q_1}(\ell))} *\dotsb * P_{\q_s}^{(12v_{\q_s}(\ell))},
\end{equation*}
where $P_{\q}$ is defined as 
\begin{equation*}
P_\q(X) = X^2 - a_\q X + \Norm(\q),
\end{equation*}
and where as usual $a_\q = \Norm (\q) + 1 - \#E(\OK/\q)$.  Then define
the integer $B_\ell$ by
\begin{equation*}
B_\ell = \prod_{k=0}^{[\frac{d}{2}]} P_\ell^*(\ell^{12k}).
\end{equation*}
where $[\frac{d}{2}]$ denotes the integer part of $\frac{d}{2}$.  We
have the following theorem of Billerey:

\begin{theorem}[{\cite[Corollaire 2.5]{billerey:isogenies}}] Let $p
\in \ZZ$ be a prime such that $E$ admits a $p$-isogeny defined over
$K$. Then one of the following is true:
\begin{enumerate}
\item the prime $p$ divides $6 \Delta_K N_{K/\Q}(\Delta_E)$; or
\item for all primes $\ell$, the number $B_\ell$ is divisible by $p$
(if $K = \Q$, we consider only $\ell \neq p$).
\end{enumerate}
\end{theorem}

\begin{remark}
The above criterion is effectively useful only if not all of the
$B_\ell$'s are zero. This is the case for number fields of odd degree
\cite[Corollary 0.2]{billerey:isogenies}.  We note that Billerey gives
a similar criterion for the even degree case.
\end{remark}

\subsection{} \label{sec:isogenies2} Let $K$ be a number field of odd
degree and $E/K$ an elliptic curve without complex multiplication over
$K$ given by a Weierstrass equation with coefficients in $\OK$.  The
following algorithm then outputs a provably finite set of primes
containing $\Red(E/K)$, the set of primes $p$ such that $E$ has a
$p$-isogeny (i.e., such that the Galois representation is reducible).

\begin{enumerate}
\item Compute the set $S_1$ of prime divisors of $6 \Delta_K N_{K/\Q}(\Delta_E)$.
\item Let $\ell_0$ be the smallest prime number not in $S_1$. The
curve $E$ has good reduction at $\ell_0$. If $B_{\ell_0} \neq 0$,
proceed to the next step. Otherwise, reiterate this step with the
smallest prime number $\ell_1$ not in $S_1$ and such that $\ell_1 >
\ell_0$ etc.~until we have some $B_\ell \neq 0$.
\item We now have a non-zero integer $B_\ell$. For greater efficiency,
we can reiterate step 2 to obtain more such $B_\ell \neq 0$. We then
define $S_2$ to be the set of prime factors of the greatest common
divisor of the $B_\ell$'s we have obtained and define $S = S_1 \cup
S_2$.
\item The set $S$ then contains $\Red(E/K)$, although it may contain
other primes. We can eliminate some of these primes by calculating
polynomials $P_\q$ for some prime ideals $\q$ of good reduction --- in
particular, if $P_\q$ is irreducible modulo $p$ (with $\q$ not
dividing~$p$), then $p \notin \Red(E/K)$. The subset $S'$ of $S$ of
prime numbers remaining is then usually small.
\end{enumerate}

Now let $K$ be our cubic number field $F$. Note that CM isogenies are
defined over imaginary quadratic fields. Since $F$ contains no such
subfield, there are no CM isogenies defined over $F$. Therefore, by
using this algorithm in combination with V{\'e}lu's formulas, we can
find representatives of all isomorphisms in a given isogeny class of
elliptic curves over $F$.

\begin{example}\label{ex:12}
Consider the curve $E$ with Weierstrass coefficients $[a^2
  + 1,-a^2 + a - 1,0,1,0]$.  The discriminant of $E$ is $\Delta_E
= 12a^2 - 25a - 43$, and 
\[\Norm_{F/\Q}(\Delta_E) = -67375 = 5^3 \cdot 7^2 \cdot 11.\]  
Thus $S_1 = \set{2, 3, 5, 7, 11, 23 }$.  Computing
$B_\ell$ for $\ell \in \set{13, 17, 19, 29}$, we see that the greatest
common divisor of the $B_\ell$ is $2^{16} \cdot 3^9$.  Then $S_2 =
\set{2,3}$, and so $S = S_{1} = \set{2,3,5,7,11,23}$.  Let $\p_2$ denote the
prime above $2$.  Then $P_{\p_2}(x) = x^2 + 3x + 8$ is irreducible
modulo $5$, $7$, and $11$.  Let $\p_{17}$ denote the degree $1$ prime
above $17$.  Then $P_{\p_{17}}(x) = x^2 + 6x + 17$ is irreducible
modulo $23$.  It follows that $\Red(E/F) \subseteq \set{2,3}$.  Using
 V\'elu's formulas, we compute $2$ and $3$-isogenies of $E$ and all
resulting curves until we get a set of elliptic curves which is closed
under $2$ and $3$-isogenies, up to isomorphism.  This computation
yields a set of $12$ representatives for the isomorphism classes of
elliptic curves in the isogeny class of $E$.  This is the unique
isogeny class of norm conductor $385$ (label 140a).  The prime isogeny graph
is shown in Figure~\ref{fig:isogeny}. 
\begin{figure}
\includegraphics[scale=0.5]{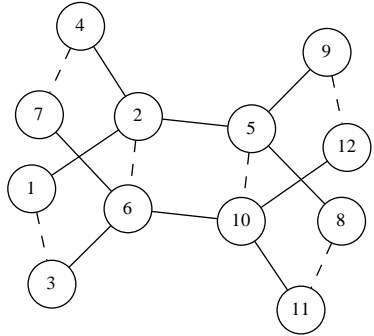}      
  \caption{Prime isogeny graph for elliptic curve of norm conductor $385$. The solid lines represent $2$-isogenies, and the dashed lines represent $3$-isogenies. }\label{fig:isogeny}
\end{figure}

\end{example}

\section{Results \& tables} \label{section 5}

\subsection{} We computed $H^4(\Gamma_{0}(\n); \tilde{\Omega}_{\C})$
for the first $4246$ levels, ordered by norm.  This includes all of
the levels of norm less than $11575$ and three of the ideals of norm
$11575$.  The current bottleneck preventing further computation
performing linear algebra on large sparse matrices.  Because of this,
we expect to be able to push the computation further in special
families, such as congruence subgroups of prime level.

Of these $4246$ levels, Heuristic~\ref{conj:eis} implies that $1492$
have nontrivial cuspidal cohomology.  Of these,
Heuristic~\ref{conj:oldform} implies that $1175$ have a nontrivial
newspace.  We found elliptic curves of matching conductor at $1020$ of
these levels, accounting for the full newspace in all but $213$
levels.  These elliptic curves comprise our dataset $\D$, of which we
provide a sample in Appendix~\ref{appendix}.  Of the remaining $213$
levels, one falls within the range of our Hecke computations, and we
can see that it corresponds to the base change of the classical weight
two newform of level $23$ with eigenvalues in $\Q (\sqrt{5})$
(cf.~\cite[\S9]{gunnells-yasaki12}). 

\subsection{}
This leaves $212$ levels with unexplained cuspidal cohomology.  We
note, however, that for each of these $212$ levels, the
cohomology that is left has rank $2$ or greater; in particular there
were no predicted new cuspidal subspaces of dimension $1$ (according
to our heuristics) for which we could not find a corresponding
elliptic curve.  This constitutes circumstantial evidence that our list
of elliptic curves over $F$ of norm conductor less than $11575$ may be
complete. That is to say, there is no clear reason (based on all the
information now at hand) to predict another curve at any of these levels.

We judge it very likely that no elliptic curves are missing from our list.
We conclude this on the basis of detailed information obtained from the
search algorithms, and other circumstantial evidence.  All the curves were
found with a certain level of effort, and searching with substantially more
effort produces no more curves.  On close examination of the output, it seems
likely that on the auxiliary curves $E_w$, all {\it integral} points, and
all points of height small enough to be relevant, were found in the searches.
If the conductors were substantially larger, one would be less confident;
eventually there must certainly exist curves that would require much more
effort to find using these methods.  A curve that is missing would be likely
to be an interesting curve with some unusual properties, such as large height.

In the course of computing the elliptic curves in $\D$, we encountered
curves whose discriminant norm was small (less than $100000$), but
whose conductor lay outside the limits of our cohomology computations.
These curves, together with the curves in $\D$, comprise a larger set
of elliptic curves over $F$ of small conductor. Partial data can be
downloaded from \cite{klages-mundt12} (as well as data for elliptic
curves over other nonreal cubic fields).  The complete larger dataset
is available via the $L$-functions and Modular Forms Database
(\url{http://www.lmfdb.org/}) \cite{lmfdb}.

\subsection{}
In the remainder of this section, we provide tables summarizing our
computations, and other highlights of the data. In all tables, only
elliptic curves from $\D$ are included.  In these tables, \#isom
refers to the number of isomorphism classes, \#isog refers to the
number of isogeny classes, $\n$ and $\Norm(\n)$ refer respectively to
the conductor and norm conductor of a given elliptic curve.  We encode
Weierstrass equations as  vectors of coefficients: $[a_1,a_2,a_3,a_4,a_6]$.

Table~\ref{tab:ecranks} gives the number of isogeny classes and
isomorphism classes in $\D$ that we found, sorted by algebraic rank.
Note that, in a few cases, \magma gave an upper and lower bound on the
rank that were not equal.  In those instances, we switched to an
isogenous curve and recomputed to try to get a larger lower bound or
smaller upper bound.  This was successful for every curve in our
dataset.  The first rank one elliptic curve we found occurs at norm
conductor $719$, and and the first rank two curve occurs at norm
conductor $9173$.  For every curve in $\D$, we found the algebraic
rank agreed with the analytic rank, where analytic rank was computed
by \magma.  The algorithm used is heuristic, numerically computing
derivatives of the $L$-function $L(E,s)$ at $s=1$ until one appears to
be nonzero.
\begin{table}
\caption{\label{tab:ecranks}Elliptic curves over $F$}
\begin{tabular}{c r r r}
\toprule
\bf{rank} &  \bf{\#isog} &  \bf{\#isom} &  \bf{smallest} $\Norm(\n)$ \\
\midrule
$0$    &  506 &  1729 &  $89$ \\
$1$    &  812 &  1483 &  $719$ \\
$2$    &  8 & 9 &  9173 \\
\midrule
total  & 1326 & 3221  &  \\
\bottomrule
\end{tabular}
\end{table}

In Table~\ref{tab:numisog} we give the sizes of isogeny classes and
the number of isogeny classes of each size in $\D$.  We find some
isogeny classes of cardinality $12$, which is larger than the
cardinalities observed over $\Q$ and $\Q(\sqrt{5})$ (see
\cite{bdkm+12}).  The computation of one such class is described in
Example~\ref{ex:12}; the other class appears in
Appendix~\ref{appendix} at label 247a (norm conductor 665).

\begin{table}
\caption{\label{tab:numisog}Number of isogeny classes of a given size}
\begin{tabular}{|c| | c | c | c | c | c | c | c | c | }
\hline
\textbf{size} &1 & 2 & 3 & 4 & 6 & 8 & 10 & 12  \\
\hline
\textbf{number}& 645 & 634 & 64 & 484 & 82 & 70 & 1 & 2  \\
\hline
\end{tabular}
\end{table}

Table~\ref{tab:primeisog} gives the number of isogeny classes and the
number of isomorphism classes with isogenies of each prime degree that
we encountered. Note that these may not represent all possible prime
degrees of isogenies over $F$. We also provide an example curve, which
need not have minimal norm conductor, that exhibits an isogeny of the
given degree.  

\begin{table}
\caption{\label{tab:primeisog}Prime isogeny degrees}
\begin{tabular}{c r r@{\hspace{5ex}} l r}
\toprule
\bf{degree}  & \bf{\#isog} & \bf{\#isom} & \bf{example curve} & $\Norm(\n)$ \\
\midrule
None &  754 & 754 &  $[1, a^2 + a - 1, a^2 + a, -a - 1, -a^2 + 1]$ & $727$ \\
2    &  824 & 3844 &  $[a + 1, -a^2 - a - 1, a^2 + a, -a^2, -a^2 + 1]$ & $89$ \\
3    &  435 & 1452 &  $[a, a - 1, 1, -a, 0]$ & $136$ \\
5    &  86 & 232 &  $[a, -a, a^2 + a + 1, -a, -2a^2 + 1]$ & $289$ \\
7    &  30 & 72 &  $[a, -a - 1, a^2 + 1, 1, -a^2]$ & $625$ \\
\bottomrule
\end{tabular}
\end{table}

Table~\ref{tab:torsionsubgroups} gives the number of isomorphism
classes of elliptic curves with given torsion structure. Again we
include an example curve, which need not have minimal norm conductor,
realizing a given torsion group. We find examples for all torsion
subgroups that appear infinitely often over $F$, as proven in
\cite{najman12}, and no others. It is unknown whether there are other
subgroups that only appear finitely over $F$.

\begin{table}[htb]
\caption{\label{tab:torsionsubgroups}Torsion subgroups}
\begin{tabular}{l@{\hspace{1ex}} r l r}
\toprule
\bf{torsion} & \bf{\#isom} & \bf{example curve} & $\Norm(\n)$ \\
\midrule
$0 $                    & 738 & $[-a^2 + a, -a^2 + a - 1, -1, 0, 0]$ & $719$ \\ 
$\Z_2$               & 1222 & $[-a^2 + a, -a^2, a^2 - a + 1, -1, 0]$ & $817$ \\ 
$\Z_3$               & 223 & $[1, a, 0, 2a^2 - a - 3, 2a^2 - 2a - 3]$ & $773$ \\ 
$\Z_2 \times \Z_2$ & 254 & $[0, -a - 1, 0, 6a - 5, -4a^2 + 7a - 3]$ & $512$ \\ 
$\Z_4$               & 301 & $[a^2, -a^2 - 1, a^2, a^2 + 1, -a^2 + a]$ & $911$ \\ 
$\Z_5$               & 53 & $[-1, a^2 - a, a, 1, 0]$ & $289$ \\ 
$\Z_6$               & 251 & $[a, -a - 1, a^2, -a^2 + a + 1, 0]$ & $593$ \\ 
$\Z_7$               & 17 & $[a^2 - 1, -a + 1, a^2 - a + 1, 0, 0]$ & $293$ \\ 
$\Z_8$               & 29 & $[a, 1, a, 0, 0]$ & $553$ \\ 
$\Z_2 \times \Z_4$ & 77 & $[0, a^2 + 1, 0, a^2, 0]$ & $512$ \\ 
$\Z_9$               & 6 & $[0, -a, -a - 1, -a^2 - a, 0]$ & $107$ \\ 
$\Z_{10}$              & 20 & $[a - 1, -a^2 - 1, a^2 - a, a^2, 0]$ & $89$ \\ 
$\Z_{12}$              & 8 & $[a, -a^2 + a + 1, a + 1, 0, 0]$ & $185$ \\ 
$\Z_2 \times \Z_6$ & 16 & $[a, a + 1, a, 6a - 5, 4a^2 - 7a + 2]$ & $115$ \\ 
$\Z_2 \times \Z_8$ & 5 & $[a, -1, a, -5a^2 + 8a - 5, -4a^2 + 9a - 4]$ & $805$ \\ 
$\Z_2 \times \Z_{12}$ & 1 & $[a^2,-a^2-a-1,a^2+1,-4a^2+11a-5,6a^2-15a+11]$ & $385$ \\
\bottomrule
\end{tabular}
\end{table}

Finally, we consider whether any curves that we found have CM (i.e.,
whether or not $\End(E) \not \simeq \Z$).  A complete list of CM
$j$-invariants in $F$ (provided to us by Cremona) is given in
Table~\ref{tab:jinv}.  Examining the $j$-invariants of the elliptic
curves in $\D$, we see that no elliptic curve in $\D$ has CM.  At
larger levels we do see CM curves.  In particular, we find some with
CM in a quadratic order of discriminant $-3$, 
such as $[0, 0, a^{2}+a, 0, -a^{2}+1]$.

\begin{table}
  \caption{The CM $j$-invariants in $F$ with fundamental discriminant $D$ and conductor $f$.}\label{tab:jinv}
$\begin{array}{rrr}
  \toprule
D & f & j\\
\midrule
-3 & 3 & -12288000\\
-3 & 2 & 54000\\
-3 & 1 & 0\\
-4 & 2 & 287496\\
-4 & 1 & 1728\\
-7 & 2 & 16581375\\
-7 & 1 & -3375\\
-8 & 1 & 8000\\
-11 & 1 & -32768\\
-19 & 1 & -884736\\
-43 & 1 & -884736000\\
-67 & 1 & -147197952000\\
-163 & 1 & -262537412640768000\\
-23 & 2 & 3792102031375a^2 - 6654675189750a + 5023465669375\\
-23 & 1 & -1084125a^2 + 1904875a - 1437500\\
\bottomrule
\end{array}$
\end{table}

\bibliographystyle{amsplain_initials_eprint_doi_url}
\bibliography{neg23paper}
\appendix
\newpage
\iflandscapetable
\begin{landscape}
\fi

\thispagestyle{empty}
\pagestyle{empty}
\section{\texorpdfstring{Table of Elliptic Curves over $F$}{Table of Elliptic Curves over F}} \label{appendix}

\begin{center}
\footnotesize

\iflandscapetable
\begin{longtable}{c c c p{0.85\textheight} r r}
\else
\begin{longtable}{c c c p{0.5\textwidth} r r}
\fi

\toprule
\bf{label} & $\Norm(\n)$ & \bf{generator of $\n$} & \bf{Weierstrass model} & \bf{rank} & \bf{torsion} \\
\midrule
\endfirsthead
\toprule
\bf{label} & $\Norm(\n)$ & \bf{generator of $\n$} & \bf{Weierstrass model} & \bf{rank} & \bf{torsion} \\
\midrule
\endhead
\bottomrule
\endfoot
\bottomrule
\endlastfoot
$ 33a $ & 89 & $ 4a^2-a-5 $  & $ [a + 1,2a^2 + 2a + 2,2a^2 + a,8a^2 + 2a - 3,6a^2 - 2a - 5] $ & 0 & $ \ZZ_{10} $ \\*
&&&$ [a + 1,2a^2 + 2a + 2,2a^2 + a,3a^2 + 7a - 8,2a^2 - 8a + 3] $ & 0 & $ \ZZ_{10} $ \\*
&&&$ [a + 1,2a^2 + 2a + 2,2a^2 + a,-17a^2 + 72a - 63,-144a^2 + 336a - 291] $ & 0 & $ \ZZ_2 $ \\*
&&&$ [a + 1,2a^2 + 2a + 2,2a^2 + a,-22a^2 + 82a - 53,-88a^2 + 334a - 321] $ & 0 & $ \ZZ_2 $ \\
$ 40a $ & 107 & $ -5a^2+3a $  & $ [0,2a^2 + 1,-a,4a^2 + a,3a^2 - 2a - 3] $ & 0 & $ \ZZ_9 $ \\*
&&&$ [0,2a^2 + 1,-a,14a^2 + 141a + 100,-968a^2 + 444a + 887] $ & 0 & $ \ZZ_3 $ \\*
&&&$ [0,2a^2 + 1,-a,34a^2 + 51a + 20,-2515a^2 + 676a + 1943] $ & 0 & $ 0 $ \\
$ 43a $ & 115 & $ -2a^2-2a-3 $  & $ [1,2a^2 + 4,-a^2 - a,7a^2 + 4,4a^2 - a + 1] $ & 0 & $ \ZZ_{12} $ \\*
&&&$ [1,2a^2 + 4,-a^2 - a,-3a^2 + 15a - 6,4a^2 - 4a + 8] $ & 0 & $ \ZZ_2\times\ZZ_6 $ \\*
&&&$ [1,2a^2 + 4,-a^2 - a,-28a^2 + 20a + 9,-68a^2 + 67a + 84] $ & 0 & $ \ZZ_6 $ \\*
&&&$ [1,2a^2 + 4,-a^2 - a,-138a^2 + 250a - 181,916a^2 - 1607a + 1220] $ & 0 & $ \ZZ_6 $ \\*
&&&$ [1,2a^2 + 4,-a^2 - a,17a^2 - 5a + 9,19a^2 + a + 3] $ & 0 & $ \ZZ_4 $ \\*
&&&$ [1,2a^2 + 4,-a^2 - a,-8a^2 + 40a - 26,-69a^2 + 152a - 113] $ & 0 & $ \ZZ_2\times\ZZ_2 $ \\*
&&&$ [1,2a^2 + 4,-a^2 - a,-333a^2 + 550a - 396,-4452a^2 + 7789a - 5755] $ & 0 & $ \ZZ_2 $ \\*
&&&$ [1,2a^2 + 4,-a^2 - a,-83a^2 + 250a - 216,862a^2 - 1681a + 1125] $ & 0 & $ \ZZ_2 $ \\
$ 52a $ & 136 & $ 6a^2-2a-2 $  & $ [a + 1,7a^2 + a + 3,6a^2,36a^2 - 16a - 20,6a^2 - 38a - 26] $ & 0 & $ \ZZ_9 $ \\*
&&&$ [a + 1,7a^2 + a + 3,6a^2 + 2a,190a^2 - 112a - 180,-1128a^2 - 200a + 500] $ & 0 & $ \ZZ_3 $ \\*
&&&$ [a + 1,7a^2 + a + 3,6a^2,306a^2 - 156a - 280,628a^2 - 852a - 996] $ & 0 & $ 0 $ \\
$ 58a $ & 161 & $ -5a^2+5a+4 $  & $ [a^2 + a,a + 3,a - 1,594a^2 - 4a - 340,-3877a^2 - 3207a - 212] $ & 0 & $ \ZZ_4 $ \\*
&&&$ [a^2 + a,a + 3,a - 1,4a^2 + a,2a^2 - 2a - 3] $ & 0 & $ \ZZ_8 $ \\*
&&&$ [a^2 + a,a + 3,a - 1,324a^2 + 131a - 100,848a^2 + 2478a + 1351] $ & 0 & $ \ZZ_2 $ \\*
&&&$ [a^2 + a,a + 3,a - 1,39a^2 + a - 20,-28a^2 - 63a - 32] $ & 0 & $ \ZZ_2\times\ZZ_4 $ \\*
&&&$ [a^2 + a,a + 3,a - 1,44a^2 + 6a - 20,-19a^2 - 23a - 8] $ & 0 & $ \ZZ_2\times\ZZ_2 $ \\*
&&&$ [a^2 + a,a + 3,a - 1,-156a^2 - 39a + 60,-770a^2 - 64a + 389] $ & 0 & $ \ZZ_2 $ \\
$ 59a $ & 167 & $ -5a^2+3a-3 $  & $ [1,2a^2 + 4,-a^2 - a + 1,8a^2 - a + 4,6a^2 - 2a + 1] $ & 0 & $ \ZZ_7 $ \\*
&&&$ [1,2a^2 + 4,-a^2 - a + 1,-12a^2 + 59a - 21,-73a^2 + 227a - 211] $ & 0 & $ 0 $ \\
$ 70a $ & 185 & $ -a^2-5a+4 $  & $ [a^2 + a,-a^2 + 3a + 3,-1,3a^2 + 5a + 1,3a^2 + a - 1] $ & 0 & $ \ZZ_{12} $ \\*
&&&$ [a^2 + a,-a^2 + 3a + 3,-1,3a^2 + 10a - 89,7a^2 - 48a - 401] $ & 0 & $ \ZZ_6 $ \\*
&&&$ [a^2 + a,-a^2 + 3a + 3,-1,-2347a^2 + 4145a - 3209,-80439a^2 + 141063a - 106939] $ & 0 & $ \ZZ_2 $ \\*
&&&$ [a^2 + a,-a^2 + 3a + 3,-1,18a^2 + 15a - 29,-34a^2 + 9a - 39] $ & 0 & $ \ZZ_4 $ \\*
&&&$ [a^2 + a,-a^2 + 3a + 3,-1,3a^2 + 5a - 4,3a^2 - 3a - 10] $ & 0 & $ \ZZ_2\times\ZZ_6 $ \\*
&&&$ [a^2 + a,-a^2 + 3a + 3,-1,-122a^2 + 260a - 214,-1280a^2 + 2192a - 1688] $ & 0 & $ \ZZ_2\times\ZZ_2 $ \\*
&&&$ [a^2 + a,-a^2 + 3a + 3,-1,3a^2 + 1,-a^2 + 6a - 15] $ & 0 & $ \ZZ_{12} $ \\*
&&&$ [a^2 + a,-a^2 + 3a + 3,-1,-137a^2 + 295a - 179,-1445a^2 + 2353a - 1473] $ & 0 & $ \ZZ_4 $ \\
$ 85a $ & 223 & $ -5a^2+3a-4 $  & $ [a^2,2a^2 + 2a + 3,-a - 1,-5a^2 + 37a - 52,53a^2 - 136a + 94] $ & 0 & $ \ZZ_4 $ \\*
&&&$ [a^2,2a^2 + 2a + 3,2a^2 - a - 1,9a^2 + 3a - 1,8a^2 - a - 5] $ & 0 & $ \ZZ_8 $ \\*
&&&$ [a^2,2a^2 + 2a + 3,-a - 1,10a^2 + 2a - 7,5a^2 - 9a - 9] $ & 0 & $ \ZZ_2\times\ZZ_4 $ \\*
&&&$ [a^2,2a^2 + 2a + 3,-a - 1,270a^2 - 598a - 602,2135a^2 - 8720a - 7783] $ & 0 & $ \ZZ_2 $ \\*
&&&$ [a^2,2a^2 + 2a + 3,-a - 1,25a^2 - 33a - 42,9a^2 - 190a - 148] $ & 0 & $ \ZZ_2\times\ZZ_2 $ \\*
&&&$ [a^2,2a^2 + 2a + 3,-a - 1,20a^2 - 28a - 42,19a^2 - 184a - 169] $ & 0 & $ \ZZ_2 $ \\
$ 92a $ & 253 & $ 7a^2-5a-5 $  & $ [a^2 + a,2a + 3,a,179a^2 - 83a - 170,1403a^2 - 497a - 1172] $ & 0 & $ \ZZ_4 $ \\*
&&&$ [a^2 + a,2a + 3,a,4a^2 + 2a,4a^2 - a - 3] $ & 0 & $ \ZZ_8 $ \\*
&&&$ [a^2 + a,2a + 3,a,14a^2 - 3a - 10,32a^2 - 18a - 32] $ & 0 & $ \ZZ_2\times\ZZ_4 $ \\*
&&&$ [a^2 + a,2a + 3,a,9a^2 - 3a - 10,33a^2 - 7a - 28] $ & 0 & $ \ZZ_2\times\ZZ_2 $ \\*
&&&$ [a^2 + a,2a + 3,a,-36a^2 - 78a - 90,142a^2 + 503a + 127] $ & 0 & $ \ZZ_2 $ \\*
&&&$ [a^2 + a,2a + 3,a,-26a^2 + 72a + 70,388a^2 + 107a - 147] $ & 0 & $ \ZZ_2 $ \\
$ 94a $ & 259 & $ 4a^2-7a-1 $  & $ [0,2a^2 + 2a,-a^2 - a,5a^2 - 3a - 5,-4a^2 - 3a] $ & 0 & $ \ZZ_9 $ \\*
&&&$ [0,2a^2 + 2a,-a^2 - a,1715a^2 + 1167a - 5225,55166a^2 + 51300a - 133449] $ & 0 & $ 0 $ \\*
&&&$ [0,2a^2 + 2a,-a^2 - a,-5a^2 + 17a + 15,-24a^2 + 10a + 21] $ & 0 & $ \ZZ_9 $ \\*
&&&$ [0,2a^2 + 2a,-a^2 - a,205a^2 - 33a - 205,-1334a^2 - 265a + 374] $ & 0 & $ \ZZ_3 $ \\
$ 101a $ & 275 & $ 8a^2-2a-3 $  & $ [a^2 + 1,a^2 + 3a + 2,-a,-81a^2 - 4a + 40,-514a^2 + 290a + 509] $ & 0 & $ \ZZ_2 $ \\*
&&&$ [a^2,4a^2 + 3a + 1,a^2 - a - 1,18a^2 - 4a - 13,3a^2 - 11a - 10] $ & 0 & $ \ZZ_{10} $ \\*
&&&$ [a^2 + 1,a^2 + 3a + 2,-a,-96a^2 + 21a + 25,-558a^2 + 346a + 448] $ & 0 & $ \ZZ_2 $ \\*
&&&$ [a^2,4a^2 + 3a + 1,a^2 - a - 1,-7a^2 - 9a - 3,-118a^2 + 58a + 111] $ & 0 & $ \ZZ_{10} $ \\*
&&&$ [a^2 + a,-a^2 + 2a + 4,0,-3a - 1,-8a^2 - 17a - 8] $ & 0 & $ \ZZ_{10} $ \\*
&&&$ [a^2 + a,-a^2 + 2a + 4,0,-25a^2 - 8a + 9,39a^2 - 30a - 45] $ & 0 & $ \ZZ_{10} $ \\
$ 105a $ & 289 & $ 3a^2-7a-2 $  & $ [a^2,2a^2 + 3a + 1,-a - 1,9a^2 - 5,3a^2 - 3a - 4] $ & 0 & $ \ZZ_5 $ \\*
&&&$ [a^2,2a^2 + 3a + 1,-a - 1,14a^2 + 15a - 15,33a^2 + 45a - 15] $ & 0 & $ 0 $ \\
$ 107a $ & 293 & $ -5a^2-2a-2 $  & $ [a^2 + a,a^2 + 2a + 5,a^2 + a - 1,9a^2 + 5a + 4,12a^2 + a - 4] $ & 0 & $ \ZZ_7 $ \\*
&&&$ [a^2 + a,a^2 + 2a + 5,a^2 + a - 1,24a^2 - 5a - 21,-46a^2 - 71a - 48] $ & 0 & $ 0 $ \\
$ 128a $ & 344 & $ 6a^2-2a-8 $  & $ [a,4a^2 + 2a,2a^2 + 2a + 2,-98128a^2 + 37792a + 82728,1108440a^2 - 10880182a - 8872784] $ & 0 & $ 0 $ \\*
&&&$ [a,4a^2 + 2a,2a^2 + 2a + 2,12a^2 - 8a - 12,-10a^2 - 12a - 4] $ & 0 & $ \ZZ_7 $ \\*
&&&$ [a,4a^2 + 2a,2a^2 + 2a + 2,42a^2 + 42a + 8,584a^2 - 432a - 660] $ & 0 & $ \ZZ_7 $ \\
$ 132a $ & 359 & $ 7a^2-6a-2 $  & $ [1,2a + 3,-a^2 - a,3a^2 + 5a + 3,3a^2 + 2a] $ & 0 & $ \ZZ_6 $ \\*
&&&$ [1,2a + 3,-a^2 - a,18a^2 + 10a - 2,25a^2 - 21a - 30] $ & 0 & $ \ZZ_6 $ \\*
&&&$ [1,2a + 3,-a^2 - a,53a^2 - 10a - 37,245a^2 - 31a - 163] $ & 0 & $ \ZZ_2 $ \\*
&&&$ [1,2a + 3,-a^2 - a,53a^2 - 5a - 37,253a^2 - 46a - 184] $ & 0 & $ \ZZ_2 $ \\
$ 140a $ & 385 & $ -6a^2+7a+5 $  & $ [a^2 + 1,2a^2 + a + 2,-a^2 - a,9a^2 - a - 3,3a^2 - 2a - 2] $ & 0 & $ \ZZ_{12} $ \\*
&&&$ [a^2 + 1,2a^2 + a + 2,-a^2 - a,-11a^2 + 24a + 27,100a^2 + 79a - 2] $ & 0 & $ \ZZ_4 $ \\*
&&&$ [a^2 + 1,2a^2 + a + 2,-a^2 - a,9a^2 - a - 8,-a^2 - 5a - 3] $ & 0 & $ \ZZ_2\times\ZZ_{12} $ \\*
&&&$ [a^2 + 1,2a^2 + a + 2,-a^2 - a,4a^2 + 14a - 68,-22a^2 - 76a + 148] $ & 0 & $ \ZZ_{12} $ \\*
&&&$ [a^2 + 1,2a^2 + a + 2,-a^2 - a,149a^2 - 346a - 308,508a^2 - 3446a - 2909] $ & 0 & $ \ZZ_6 $ \\*
&&&$ [a^2 + 1,2a^2 + a + 2,-a^2 - a,14a^2 - 16a - 28,-16a^2 - 66a - 58] $ & 0 & $ \ZZ_2\times\ZZ_6 $ \\*
&&&$ [a^2 + 1,2a^2 + a + 2,-a^2 - a,-5821a^2 + 6819a - 3688,-141983a^2 + 262157a - 249179] $ & 0 & $ \ZZ_2 $ \\*
&&&$ [a^2 + 1,2a^2 + a + 2,-a^2 - a,-26a^2 + 49a + 7,53a^2 + 168a - 75] $ & 0 & $ \ZZ_2\times\ZZ_4 $ \\*
&&&$ [a^2 + 1,2a^2 + a + 2,-a^2 - a,-41a^2 + 74a - 68,-220a^2 + 350a - 467] $ & 0 & $ \ZZ_6 $ \\*
&&&$ [a^2 + 1,2a^2 + a + 2,-a^2 - a,-351a^2 + 429a - 233,-2409a^2 + 4504a - 4046] $ & 0 & $ \ZZ_2\times\ZZ_2 $ \\*
&&&$ [a^2 + 1,2a^2 + a + 2,-a^2 - a,59a^2 + 69a - 73,307a^2 + 308a + 124] $ & 0 & $ \ZZ_4 $ \\*
&&&$ [a^2 + 1,2a^2 + a + 2,-a^2 - a,-81a^2 + 119a - 618,-2523a^2 + 775a - 6857] $ & 0 & $ \ZZ_2 $ \\
$ 145a $ & 392 & $ -8a^2+6a+6 $  & $ [a^2 + 1,2a^2 + 2a,2a,-1154a^2 + 2028a - 1540,-27332a^2 + 47956a - 36202] $ & 0 & $ 0 $ \\*
&&&$ [1,3a^2 + a + 2,2a^2,8a^2 - 4,4a^2 - 2a - 4] $ & 0 & $ \ZZ_7 $ \\*
&&&$ [a + 1,2a^2 + a + 4,4a^2 + 2,6a^2 + 2a + 2,-4a^2 + 12a - 8] $ & 0 & $ \ZZ_7 $ \\
$ 163a $ & 440 & $ 8a^2+2a-6 $  & $ [a^2,4a^2 + a + 3,2a^2 + 2a + 2,14a^2 - 4a - 4,4a^2 - 12a - 8] $ & 0 & $ \ZZ_6 $ \\*
&&&$ [a^2,4a^2 + a + 3,2a^2 + 2a + 2,24a^2 - 4a - 14,32a^2 - 24a - 42] $ & 0 & $ \ZZ_6 $ \\*
&&&$ [a^2,4a^2 + a + 3,2,6a^2 - 16,-24a^2 + 28a - 24] $ & 0 & $ \ZZ_2 $ \\*
&&&$ [a^2,4a^2 + a + 3,2,-74a^2 + 160a - 136,-736a^2 + 1364a - 1072] $ & 0 & $ \ZZ_2 $ \\
$ 168a $ & 449 & $ a^2-8a $  & $ [a^2 + a,a^2 + 3a + 3,a^2,9a^2 + 4a - 3,8a^2 - 2a - 7] $ & 0 & $ \ZZ_6 $ \\*
&&&$ [a^2 + a,a^2 + 3a + 3,a^2,4a^2 + 9a - 8,5a^2 - 10a + 2] $ & 0 & $ \ZZ_6 $ \\*
&&&$ [a^2 + a,a^2 + 3a + 3,a^2,-21a^2 + 59a - 43,-120a^2 + 232a - 186] $ & 0 & $ \ZZ_2 $ \\*
&&&$ [a^2 + a,a^2 + 3a + 3,a^2,-31a^2 + 59a - 38,-133a^2 + 231a - 174] $ & 0 & $ \ZZ_2 $ \\
$ 181a $ & 475 & $ -4a^2-7a $  & $ [0,2a^2 + 2,-a,5a^2 - 2a - 1,a^2 - 2a - 1] $ & 0 & $ \ZZ_5 $ \\*
&&&$ [0,2a^2 + 2,-a,-5a^2 + 28a + 19,91a^2 + 34a - 43] $ & 0 & $ 0 $ \\
$ 185a $ & 503 & $ a^2-a-8 $  & $ [a^2 + a + 1,a^2 + 4a + 4,3a^2 + a - 1,15a^2 + 10a,24a^2 - 2a - 15] $ & 0 & $ \ZZ_6 $ \\*
&&&$ [a^2 + a + 1,a^2 + 4a + 4,3a^2 + a - 1,-20a^2 + 30a + 35,-50a^2 + 19a + 43] $ & 0 & $ \ZZ_6 $ \\*
&&&$ [a^2 + a + 1,a^2 + 4a + 4,3a^2 + a - 1,-10a^2 - 95a - 65,-717a^2 - 670a - 97] $ & 0 & $ \ZZ_2 $ \\*
&&&$ [a^2 + a + 1,a^2 + 4a + 4,3a^2 + a - 1,5a^2 - 90a - 70,-850a^2 - 696a - 41] $ & 0 & $ \ZZ_2 $ \\
$ 186a $ & 505 & $ -8a+1 $  & $ [a^2 + a,2a + 5,-1,6a^2 + 5a + 4,6a^2 + 2a - 1] $ & 0 & $ \ZZ_6 $ \\*
&&&$ [a^2 + a,2a + 5,-1,11a^2 - 1,5a^2 - 5a - 4] $ & 0 & $ \ZZ_6 $ \\*
&&&$ [a^2 + a,2a + 5,-1,-4a^2 + 20a - 1,-14a^2 + 47a - 22] $ & 0 & $ \ZZ_2 $ \\*
&&&$ [a^2 + a,2a + 5,-1,16a^2 + 15a - 16,40a^2 + 42a - 57] $ & 0 & $ \ZZ_2 $ \\
$ 187a $ & 505 & $ -2a^2-7a+2 $  & $ [a^2 + a,a^2 + 3a + 4,a^2 - 1,12a^2 + 7a,18a^2 - 10] $ & 0 & $ \ZZ_{10} $ \\*
&&&$ [a^2 + a,a^2 + 3a + 4,a^2 - 1,-53a^2 + 67a - 10,-277a^2 + 287a - 39] $ & 0 & $ \ZZ_2 $ \\*
&&&$ [a^2 + a,a^2 + 3a + 4,a^2 - 1,7a^2 + 2a,2a^2 - 5a - 4] $ & 0 & $ \ZZ_{10} $ \\*
&&&$ [a^2 + a,a^2 + 3a + 4,a^2 - 1,192a^2 - 13a - 210,-768a^2 - 130a - 74] $ & 0 & $ \ZZ_2 $ \\
$ 189a $ & 512 & $ 8a^2-8 $  & $ [-2a^2 + 2a,4a^2 + a - 2,4a + 6,8a^2 - 10a - 10,-8a^2 - 12a - 8] $ & 0 & $ \ZZ_2\times\ZZ_4 $ \\*
&&&$ [-2a^2 + 4a + 2,8a^2 - 6a - 2,8a^2 + 12a,-24a^2 - 24a - 16,-144a^2 + 80a + 16] $ & 0 & $ \ZZ_2\times\ZZ_2 $ \\*
&&&$ [-2a^2 + 2a,4a^2 - 2a + 4,-4a^2 + 8a + 8,16a^2 - 8a - 8,16a^2 - 32a - 32] $ & 0 & $ \ZZ_8 $ \\*
&&&$ [2a + 2,4a - 2,4a^2 + 12a + 4,-24a,-64a^2 - 16a] $ & 0 & $ \ZZ_4 $ \\*
&&&$ [-2a^2 + 4a + 2,8a^2 + 4,8a^2 + 16a,8a^2 - 8a - 8,-128a^2 + 96a] $ & 0 & $ \ZZ_2 $ \\*
&&&$ [-2a^2 + 4a + 2,8a^2 + 4,8a^2 + 16a + 8,-104a^2 + 216a - 216,-1760a^2 + 2880a - 2240] $ & 0 & $ \ZZ_2 $ \\
$ 202a $ & 553 & $ 9a^2-4a-2 $  & $ [a,3a^2 + 1,a^2 - a - 1,11a^2 - 3a - 88,-112a^2 - 9a + 267] $ & 0 & $ \ZZ_4 $ \\*
&&&$ [a,3a^2 + 1,a^2 - a - 1,6a^2 - 3a - 3,a^2 - 3a - 2] $ & 0 & $ \ZZ_8 $ \\*
&&&$ [a,3a^2 + 1,a^2 - a - 1,6a^2 - 3a - 8,-5a^2 - 3a + 1] $ & 0 & $ \ZZ_2\times\ZZ_4 $ \\*
&&&$ [a,3a^2 + 1,a^2 - a - 1,a^2 - 3a - 8,-22a^2 + 3a + 7] $ & 0 & $ \ZZ_2\times\ZZ_2 $ \\*
&&&$ [a,3a^2 + 1,a^2 - a - 1,-44a^2 - 83a - 3,-615a^2 - 105a + 294] $ & 0 & $ \ZZ_2 $ \\*
&&&$ [a,3a^2 + 1,a^2 - a - 1,-34a^2 + 77a - 13,-177a^2 + 235a - 156] $ & 0 & $ \ZZ_2 $ \\
$ 214a $ & 593 & $ 8a^2-a-9 $  & $ [a^2 + 1,a^2 + 2a + 2,-1,6a^2 + a - 1,2a^2 - a - 2] $ & 0 & $ \ZZ_6 $ \\*
&&&$ [a^2 + 1,a^2 + 2a + 2,-1,-4a^2 + a + 4,-a - 1] $ & 0 & $ \ZZ_6 $ \\*
&&&$ [a^2 + 1,a^2 + 2a + 2,-1,26a^2 + 6a - 11,20a^2 - 42a - 46] $ & 0 & $ \ZZ_2 $ \\*
&&&$ [a^2 + 1,a^2 + 2a + 2,-1,41a^2 - 9a - 31,121a^2 - 101a - 148] $ & 0 & $ \ZZ_2 $ \\
$ 217a $ & 595 & $ 11a^2-4a-6 $  & $ [a^2 + a,a^2 + 2a + 5,a^2,143a^2 - 81a - 146,1071a^2 - 672a - 1117] $ & 0 & $ \ZZ_8 $ \\*
&&&$ [a^2 + a,a^2 + 2a + 5,a^2,8a^2 + 4a + 4,9a^2 + a - 2] $ & 0 & $ \ZZ_8 $ \\*
&&&$ [a^2 + a,a^2 + 2a + 5,a^2,18a^2 - a - 6,35a^2 - 25a - 37] $ & 0 & $ \ZZ_2\times\ZZ_8 $ \\*
&&&$ [a^2 + a,a^2 + 2a + 5,a^2,-12177a^2 + 20079a + 21134,-1762061a^2 - 394058a + 693731] $ & 0 & $ \ZZ_2 $ \\*
&&&$ [a^2 + a,a^2 + 2a + 5,a^2,53a^2 - a - 26,3a^2 - 102a - 77] $ & 0 & $ \ZZ_2\times\ZZ_8 $ \\*
&&&$ [a^2 + a,a^2 + 2a + 5,a^2,708a^2 - 66a - 446,-5205a^2 - 3555a + 288] $ & 0 & $ \ZZ_8 $ \\*
&&&$ [a^2 + a,a^2 + 2a + 5,a^2,-42a^2 + 64a + 74,-677a^2 + 103a + 458] $ & 0 & $ \ZZ_2\times\ZZ_4 $ \\*
&&&$ [a^2 + a,a^2 + 2a + 5,a^2,-752a^2 + 1259a + 1324,-27195a^2 - 3875a + 12311] $ & 0 & $ \ZZ_2\times\ZZ_2 $ \\*
&&&$ [a^2 + a,a^2 + 2a + 5,a^2,-852a^2 - 91a + 424,-12619a^2 + 6941a + 12425] $ & 0 & $ \ZZ_4 $ \\*
&&&$ [a^2 + a,a^2 + 2a + 5,a^2,-687a^2 + 1559a + 1514,-10301a^2 - 7264a + 123] $ & 0 & $ \ZZ_2 $ \\
$ 233a $ & 625 & $ 8a^2+3a+1 $  & $ [a^2 + a + 1,a^2 + 4a + 3,3a^2 + a - 1,14a^2 + 5a - 4,14a^2 - 8a - 14] $ & 0 & $ \ZZ_5 $ \\*
&&&$ [a^2,3a^2 + 2a + 2,a^2 - a,5a^2 + 2a - 2,a^2 - 1] $ & 0 & $ \ZZ_5 $ \\*
&&&$ [a^2,3a^2 + 2a + 2,a^2 - a,-825a^2 - 158a + 353,-12899a^2 + 5980a + 11869] $ & 0 & $ 0 $ \\*
&&&$ [a^2 + a + 1,a^2 + 4a + 3,3a^2 + a - 1,-36a^2 + 55a + 61,184a^2 + 287a + 111] $ & 0 & $ 0 $ \\
$ 243a $ & 649 & $ 8a^2+a-8 $  & $ [a^2 + a + 1,2a + 3,3a^2 - 2,3a^2 + 7a + 2,7a^2 + 2a - 3] $ & 0 & $ \ZZ_6 $ \\*
&&&$ [a^2 + a + 1,2a + 3,3a^2 - 2,8a^2 + 2a + 7,16a^2 - 5a - 2] $ & 0 & $ \ZZ_6 $ \\*
&&&$ [a^2 + a + 1,2a + 3,3a^2 - 2,-32a^2 + 62a - 38,-142a^2 + 257a - 191] $ & 0 & $ \ZZ_2 $ \\*
&&&$ [a^2 + a + 1,2a + 3,3a^2 - 2,-27a^2 + 57a - 38,-153a^2 + 276a - 221] $ & 0 & $ \ZZ_2 $ \\
$ 247a $ & 665 & $ 9a^2+a-8 $  & $ [1,2a + 3,-a^2 - a,8a + 1,a^2 + 2a + 2] $ & 0 & $ \ZZ_6 $ \\*
&&&$ [1,2a + 3,-a^2 - a,-25a^2 + 53a - 34,87a^2 - 147a + 112] $ & 0 & $ \ZZ_2\times\ZZ_6 $ \\*
&&&$ [1,2a + 3,-a^2 - a,-440a^2 + 783a - 584,6190a^2 - 10854a + 8194] $ & 0 & $ \ZZ_6 $ \\*
&&&$ [1,2a + 3,-a^2 - a,-10a^2 + 43a - 44,148a^2 - 196a + 90] $ & 0 & $ \ZZ_6 $ \\*
&&&$ [1,2a + 3,-a^2 - a,-4005a^2 + 7013a - 5269,-176412a^2 + 309544a - 233665] $ & 0 & $ \ZZ_2 $ \\*
&&&$ [1,2a + 3,-a^2 - a,-50a^2 + 93a - 59,-269a^2 + 477a - 347] $ & 0 & $ \ZZ_6 $ \\*
&&&$ [1,2a + 3,-a^2 - a,-8190a^2 + 8408a - 1889,-159767a^2 + 214454a - 309664] $ & 0 & $ \ZZ_2 $ \\*
&&&$ [1,2a + 3,-a^2 - a,-4250a^2 + 7093a - 5069,-176459a^2 + 308287a - 234587] $ & 0 & $ \ZZ_2\times\ZZ_2 $ \\*
&&&$ [1,2a + 3,-a^2 - a,-60a^2 + 108a - 69,-189a^2 + 332a - 238] $ & 0 & $ \ZZ_2\times\ZZ_6 $ \\*
&&&$ [1,2a + 3,-a^2 - a,-495a^2 + 803a - 544,6075a^2 - 10829a + 8078] $ & 0 & $ \ZZ_6 $ \\*
&&&$ [1,2a + 3,-a^2 - a,215a^2 - 347a + 246,-973a^2 + 1753a - 1378] $ & 0 & $ \ZZ_6 $ \\*
&&&$ [1,2a + 3,-a^2 - a,-4230a^2 + 7058a - 5049,-178139a^2 + 311112a - 236758] $ & 0 & $ \ZZ_2 $ \\
$ 254a $ & 685 & $ -7a^2+5a-7 $  & $ [a^2 + 1,a^2 + a,-a^2 - a,179a^2 - 96a - 169,1188a^2 - 457a - 1022] $ & 0 & $ \ZZ_4 $ \\*
&&&$ [a^2 + 1,a^2 + a,-a^2 - a,4a^2 - a - 4,-a^2 - a] $ & 0 & $ \ZZ_8 $ \\*
&&&$ [a^2 + 1,a^2 + a,-a^2 - a,14a^2 - 6a - 14,16a^2 - 14a - 20] $ & 0 & $ \ZZ_2\times\ZZ_4 $ \\*
&&&$ [a^2 + 1,a^2 + a,-a^2 - a,-2196a^2 + 2034a - 1514,-20218a^2 + 56780a - 54792] $ & 0 & $ \ZZ_2 $ \\*
&&&$ [a^2 + 1,a^2 + a,-a^2 - a,9a^2 + 4a - 19,12a^2 - 3a - 38] $ & 0 & $ \ZZ_2\times\ZZ_4 $ \\*
&&&$ [a^2 + 1,a^2 + a,-a^2 - a,-136a^2 + 129a - 94,-357a^2 + 953a - 812] $ & 0 & $ \ZZ_2\times\ZZ_2 $ \\*
&&&$ [a^2 + 1,a^2 + a,-a^2 - a,74a^2 + 39a - 24,65a^2 - 335a - 336] $ & 0 & $ \ZZ_4 $ \\*
&&&$ [a^2 + 1,a^2 + a,-a^2 - a,-396a^2 + 224a + 126,-1212a^2 + 3150a + 1332] $ & 0 & $ \ZZ_2 $ \\
$ 265a $ & 712 & $ 6a^2-10a-8 $  & $ [a + 1,7a^2 + 3a + 5,4a^2 + 2a + 2,60a^2 - 2a - 24,104a^2 - 60a - 100] $ & 0 & $ \ZZ_6 $ \\*
&&&$ [a + 1,7a^2 + 3a + 5,4a^2 + 2a + 2,40a^2 + 18a - 44,72a^2 - 84a - 72] $ & 0 & $ \ZZ_6 $ \\*
&&&$ [a + 1,7a^2 + 3a + 5,4a^2 + 4a + 2,14a^2 - 8a - 4,-28a^2 - 4a + 12] $ & 0 & $ \ZZ_2 $ \\*
&&&$ [a + 1,7a^2 + 3a + 5,4a^2 + 4a + 2,-626a^2 - 8a + 316,3492a^2 - 2564a - 3892] $ & 0 & $ \ZZ_2 $ \\
$ 266a $ & 719 & $ a^2-a-9 $  & $ [a^2 + a + 1,4a + 3,2a^2 + a - 2,11a^2 + 8a,17a^2 - a - 11] $ & 0 & $ \ZZ_6 $ \\*
&&&$ [a^2 + a + 1,4a + 3,2a^2 + a - 2,6a^2 + 13a,14a^2 + a - 1] $ & 0 & $ \ZZ_6 $ \\*
&&&$ [a^2 + a + 1,4a + 3,2a^2 + a - 2,31a^2 - 7a - 20,84a^2 - 74a - 104] $ & 0 & $ \ZZ_2 $ \\*
&&&$ [a^2 + a + 1,4a + 3,2a^2 + a - 2,26a^2 - 2a - 25,82a^2 - 68a - 97] $ & 0 & $ \ZZ_2 $ \\
$ 268a $ & 719 & $ 11a^2-4a-5 $  & $ [a^2 + 1,2a^2 + 2a + 2,-a,12a^2 + a - 5,7a^2 - 7a - 9] $ & 1 & $ 0 $ \\*
&&&$ [a^2 + 1,2a^2 + 2a + 2,-a,12a^2 + a - 5,7a^2 - 7a - 9] $ & 1 & $ 0 $ \\
$ 269a $ & 721 & $ 8a^2-9 $  & $ [a^2 + a + 1,4a + 3,3a^2 + a - 1,8a^2 + 9a,14a^2 - 7] $ & 0 & $ \ZZ_6 $ \\*
&&&$ [a^2 + a + 1,4a + 3,3a^2 + a - 1,13a^2 + 4a,17a^2 - 9a - 10] $ & 0 & $ \ZZ_6 $ \\*
&&&$ [a^2 + a + 1,4a + 3,3a^2 + a - 1,-12a^2 + 34a - 20,-55a^2 + 83a - 55] $ & 0 & $ \ZZ_6 $ \\*
&&&$ [a^2 + a + 1,4a + 3,3a^2 + a - 1,-1322a^2 + 2339a - 1735,-33252a^2 + 58345a - 44010] $ & 0 & $ \ZZ_2 $ \\*
&&&$ [a^2 + a + 1,4a + 3,3a^2 + a - 1,-7a^2 + 29a - 15,-74a^2 + 119a - 83] $ & 0 & $ \ZZ_6 $ \\*
&&&$ [a^2 + a + 1,4a + 3,3a^2 + a - 1,-1487a^2 + 2539a - 1490,-35052a^2 + 58054a - 43204] $ & 0 & $ \ZZ_2 $ \\
$ 270a $ & 727 & $ 10a^2-7a-7 $  & $ [a^2 + a,a^2 + 2a + 4,-1,9a^2 + 4a - 1,8a^2 - a - 5] $ & 1 & $ 0 $ \\*
&&&$ [a^2 + a,a^2 + 2a + 4,-1,9a^2 + 4a - 1,8a^2 - a - 5] $ & 1 & $ 0 $ \\
$ 283a $ & 773 & $ -3a^2+12a-5 $  & $ [1,a + 3,1,2a^2 + a,4a^2 - 2a - 5] $ & 0 & $ \ZZ_3 $ \\*
&&&$ [1,a + 3,1,-33a^2 + 51a - 40,-138a^2 + 250a - 193] $ & 0 & $ 0 $ \\
$ 290a $ & 805 & $ 3a^2-6a-10 $  & $ [a^2,2a^2 + 3a + 3,-1,12a^2 + 5a - 1,12a^2 + a - 6] $ & 0 & $ \ZZ_8 $ \\*
&&&$ [a^2,2a^2 + 3a + 3,-1,-5278a^2 + 9825a - 7216,-281397a^2 + 497785a - 377298] $ & 0 & $ \ZZ_2 $ \\*
&&&$ [a^2,2a^2 + 3a + 3,-1,12a^2 + 5a - 6,3a^2 - 3a - 10] $ & 0 & $ \ZZ_2\times\ZZ_8 $ \\*
&&&$ [a^2,2a^2 + 3a + 3,-1,-33a^2 + 10a - 41,-326a^2 + 68a - 28] $ & 0 & $ \ZZ_2\times\ZZ_4 $ \\*
&&&$ [a^2,2a^2 + 3a + 3,-1,-388a^2 + 555a - 461,-4638a^2 + 8031a - 6106] $ & 0 & $ \ZZ_2\times\ZZ_2 $ \\*
&&&$ [a^2,2a^2 + 3a + 3,-1,-398a^2 - 455a - 181,-10630a^2 - 3651a + 3018] $ & 0 & $ \ZZ_4 $ \\*
&&&$ [a^2,2a^2 + 3a + 3,-1,57a^2 - 51,-64a^2 - 110a - 28] $ & 0 & $ \ZZ_8 $ \\*
&&&$ [a^2,2a^2 + 3a + 3,-1,-1178a^2 + 5a - 426,11693a^2 + 9629a - 14206] $ & 0 & $ \ZZ_2 $ \\
$ 291a $ & 808 & $ -6a^2-2a-4 $  & $ [a,6a^2 + a + 4,2a^2 - 2,28a^2 - 6a - 6,24a^2 - 16a - 16] $ & 0 & $ \ZZ_5 $ \\*
&&&$ [a,6a^2 + a + 4,2a^2 - 2,-52a^2 + 164a - 286,-1306a^2 + 1992a - 2502] $ & 0 & $ 0 $ \\
$ 294a $ & 809 & $ 9a^2-9a-1 $  & $ [a,4a^2 + a + 2,a^2 - 1,15a^2 - 5a - 8,7a^2 - 9a - 9] $ & 0 & $ \ZZ_5 $ \\*
&&&$ [a,4a^2 + a + 2,a^2 - 1,120a^2 + 15a - 53,-83a^2 - 503a - 331] $ & 0 & $ 0 $ \\
$ 294b $ & 809 & $ 9a^2-9a-1 $  & $ [a,3a^2 + a + 2,0,10a^2 + 2a - 1,19a^2 - 3a - 12] $ & 0 & $ \ZZ_6 $ \\*
&&&$ [a,3a^2 + a + 2,0,25a^2 - 13a - 21,47a^2 - 40a - 56] $ & 0 & $ \ZZ_6 $ \\*
&&&$ [a,3a^2 + a + 2,0,-110a^2 + 97a + 139,-454a^2 - 295a + 37] $ & 0 & $ \ZZ_2 $ \\*
&&&$ [a,3a^2 + a + 2,0,-130a^2 + 82a + 139,-843a^2 - 113a + 396] $ & 0 & $ \ZZ_2 $ \\
$ 297a $ & 817 & $ -a^2-7a-8 $  & $ [a,4a^2 + a,0,9a^2 - 6a - 9,-3a^2 - 5a - 2] $ & 1 & $ \ZZ_2 $ \\*
&&&$ [a,4a^2 + a,0,9a^2 - a - 9,4a^2 - 5a - 5] $ & 1 & $ \ZZ_2 $ \\
$ 305a $ & 829 & $ 6a^2-a-10 $  & $ [0,2,-a,1,0] $ & 1 & $ 0 $ \\*
&&&$ [0,2,-a,1,0] $ & 1 & $ 0 $ \\
$ 315a $ & 851 & $ -2a^2+10a+1 $  & $ [a^2,a + 3,a^2 - 1,a^2 + 3a + 2,a^2 + 2a - 1] $ & 0 & $ \ZZ_4 $ \\*
&&&$ [a^2,3a^2 + a + 3,-a,-138a^2 + 239a - 194,-1411a^2 + 2382a - 1788] $ & 0 & $ \ZZ_2 $ \\*
&&&$ [a^2,3a^2 + a + 3,-a,2a^2 + 14a - 14,-25a^2 + 49a - 44] $ & 0 & $ \ZZ_2\times\ZZ_2 $ \\*
&&&$ [a^2,3a^2 + a + 3,-a,-18a^2 + 29a + 6,-47a^2 + 48a - 24] $ & 0 & $ \ZZ_2 $ \\
$ 322a $ & 865 & $ 9a^2-9a-8 $  & $ [1,a^2 + 2,-a^2 - a + 1,2,-a^2 + 1] $ & 0 & $ \ZZ_6 $ \\*
&&&$ [1,a^2 + 2,-a^2 - a + 1,-35a^2 + 22,43a^2 - 49a - 61] $ & 0 & $ \ZZ_6 $ \\*
&&&$ [1,a^2 + 2,-a^2 - a + 1,-45a^2 - 40a - 3,-350a^2 - 64a + 151] $ & 0 & $ \ZZ_2 $ \\*
&&&$ [1,a^2 + 2,-a^2 - a + 1,-55a^2 - 35a + 7,-308a^2 - 124a + 82] $ & 0 & $ \ZZ_2 $ \\
$ 322b $ & 865 & $ 9a^2-9a-8 $  & $ [a,2a^2 + 2,a,a^2 + 4a - 4,-5a^2 + 11a - 10] $ & 0 & $ \ZZ_4 $ \\*
&&&$ [a,2a^2 + 2,-a,-768a^2 + 1354a - 1029,-15730a^2 + 27595a - 20831] $ & 0 & $ \ZZ_2 $ \\*
&&&$ [a,2a^2 + 2,-a,-43a^2 + 84a - 64,-283a^2 + 500a - 379] $ & 0 & $ \ZZ_2\times\ZZ_2 $ \\*
&&&$ [a,2a^2 + 2,-a,-38a^2 + 94a - 59,-288a^2 + 521a - 363] $ & 0 & $ \ZZ_4 $ \\
$ 325a $ & 875 & $ 5a^2+5a+5 $  & $ [a^2 + 1,2a^2 + 2a,-a,10a^2 - a - 7,5a^2 - 4a - 6] $ & 0 & $ \ZZ_8 $ \\*
&&&$ [a^2 + 1,2a^2 + 2a,-a,3340a^2 - 1201a - 2807,-90789a^2 - 988a + 50986] $ & 0 & $ \ZZ_4 $ \\*
&&&$ [a^2 + 1,2a^2 + 2a,-a,20a^2 - 6a - 17,-20a^2 - 19a - 3] $ & 0 & $ \ZZ_2\times\ZZ_8 $ \\*
&&&$ [a^2 + 1,2a^2 + 2a,-a,215a^2 - 76a - 182,-1389a^2 - 238a + 611] $ & 0 & $ \ZZ_2\times\ZZ_4 $ \\*
&&&$ [a^2 + 1,2a^2 + 2a,-a,-15a^2 - 16a - 12,-171a^2 - 40a + 75] $ & 0 & $ \ZZ_8 $ \\*
&&&$ [a^2 + 1,2a^2 + 2a,-a,210a^2 - 71a - 197,-1345a^2 - 304a + 492] $ & 0 & $ \ZZ_2\times\ZZ_2 $ \\*
&&&$ [a^2 + 1,2a^2 + 2a,-a,-340a^2 + 504a + 203,-8180a^2 + 3461a + 4497] $ & 0 & $ \ZZ_2 $ \\*
&&&$ [a^2 + 1,2a^2 + 2a,-a,680a^2 - 566a - 837,8966a^2 - 8093a - 11269] $ & 0 & $ \ZZ_2 $ \\
$ 325b $ & 875 & $ 5a^2+5a+5 $  & $ [a^2 + a,-a^2 + 3a + 3,a^2,-2a^2 + 5a + 5,-a^2 + 2a + 2] $ & 0 & $ \ZZ_6 $ \\*
&&&$ [a^2 + a,-a^2 + 3a + 3,a^2,-57a^2 + 10a + 40,28a^2 - 94a - 87] $ & 0 & $ \ZZ_6 $ \\*
&&&$ [a^2 + a,-a^2 + 3a + 3,a^2,73a^2 + 385a - 160,-4276a^2 + 11782a + 8122] $ & 0 & $ \ZZ_2 $ \\*
&&&$ [a^2 + a,-a^2 + 3a + 3,a^2,8a^2 - 40a - 35,-39a^2 - 216a - 141] $ & 0 & $ \ZZ_6 $ \\*
&&&$ [a^2 + a,-a^2 + 3a + 3,a^2,203a^2 - 270a - 345,-2047a^2 + 1971a + 2603] $ & 0 & $ \ZZ_2 $ \\*
&&&$ [a^2 + a,-a^2 + 3a + 3,a^2,-22a^2 - 10a,82a^2 - 365a - 327] $ & 0 & $ \ZZ_6 $ \\
$ 333a $ & 883 & $ -7a^2+4a-7 $  & $ [a^2 + a + 1,a^2 + 3a + 2,2a^2 + a - 2,10a^2 + 4a - 2,12a^2 - 2a - 9] $ & 1 & $ 0 $ \\*
&&&$ [a^2 + a + 1,a^2 + 3a + 2,2a^2 + a - 2,10a^2 + 4a - 2,12a^2 - 2a - 9] $ & 1 & $ 0 $ \\
$ 336a $ & 905 & $ -4a^2-7a+5 $  & $ [a,3a^2,-a,6a^2 - 3a - 5,-2a^2 - 4a - 2] $ & 0 & $ \ZZ_{10} $ \\*
&&&$ [a,3a^2,-a,11a^2 + 7a,a^2 + 11a + 8] $ & 0 & $ \ZZ_{10} $ \\*
&&&$ [a,3a^2,-a,-179a^2 - 253a - 115,-3512a^2 - 1301a + 947] $ & 0 & $ \ZZ_2 $ \\*
&&&$ [a,3a^2,-a,-174a^2 - 253a - 115,-3604a^2 - 1300a + 990] $ & 0 & $ \ZZ_2 $ \\
$ 338a $ & 911 & $ 11a^2-7a-7 $  & $ [1,a^2 + a + 4,-a^2 - a,266a^2 - 605a - 603,2756a^2 - 8935a - 8313] $ & 0 & $ \ZZ_2 $ \\*
&&&$ [1,a^2 + a + 4,-a^2 - a,6a^2 + 2,3a^2 - 5a - 3] $ & 0 & $ \ZZ_4 $ \\*
&&&$ [1,a^2 + a + 4,-a^2 - a,21a^2 - 35a - 33,43a^2 - 192a - 167] $ & 0 & $ \ZZ_2\times\ZZ_2 $ \\*
&&&$ [1,a^2 + a + 4,-a^2 - a,16a^2 - 25a - 23,70a^2 - 237a - 217] $ & 0 & $ \ZZ_2 $ \\
$ 351a $ & 952 & $ 10a^2+2a $  & $ [a^2 + 1,4a^2 + 6a + 5,8a^2 - 2,-1344666a^2 + 2359820a - 1781376,-1310904916a^2 + 2300477896a - 1736579460] $ & 0 & $ \ZZ_2 $ \\*
&&&$ [a^2 + 1,4a^2 + 6a + 5,4a^2 - 2,-22781142a^2 + 39978218a - 30178668,-75113763134a^2 + 131815465490a - 99504551060] $ & 0 & $ \ZZ_2 $ \\
$ 354a $ & 959 & $ a^2+4a-11 $  & $ [a + 1,4a^2 + a + 2,a^2,16a^2 - 3a - 9,8a^2 - 10a - 12] $ & 1 & $ \ZZ_2 $ \\*
&&&$ [a + 1,4a^2 + a + 2,a^2,21a^2 - 8a - 9,10a^2 - 22a - 11] $ & 1 & $ \ZZ_2 $ \\
$ 363a $ & 991 & $ 5a^2-4a-11 $  & $ [1,a + 3,-a + 1,-a^2 + 3a + 4,-3a^2 + 2a + 3] $ & 1 & $ \ZZ_3 $ \\*
&&&$ [1,a + 3,-a + 1,64a^2 - 7a - 41,-112a^2 - 81a + 2] $ & 1 & $ 0 $ \\
$ 375a $ & 1003 & $ -8a^2+10a+3 $  & $ [1,4,0,6,a^2 + a + 3] $ & 1 & $ 0 $ \\*
&&&$ [1,4,0,6,a^2 + a + 3] $ & 1 & $ 0 $ \\
$ 380a $ & 1033 & $ 12a^2-a-3 $  & $ [a,4a^2 + 2a + 2,-1,18a^2 - 2a - 11,9a^2 - 11a - 14] $ & 1 & $ \ZZ_2 $ \\*
&&&$ [a,4a^2 + 2a + 2,-1,13a^2 + 3a - 16,-9a^2 + 6a - 21] $ & 1 & $ \ZZ_2 $ \\
$ 383a $ & 1045 & $ 13a^2-5a-7 $  & $ [a^2 + a,a^2 + 2a + 3,a^2 + a,5a^2 + 2a + 1,4a^2 - 1] $ & 0 & $ \ZZ_{10} $ \\*
&&&$ [a^2 + a,a^2 + 2a + 3,a^2 + a,-10a^2 + 12a + 11,-a^2 + 42a + 32] $ & 0 & $ \ZZ_{10} $ \\*
&&&$ [a^2 + a,a^2 + 2a + 3,a^2 + a,-655a^2 + 502a - 539,-3116a^2 + 10706a - 8921] $ & 0 & $ \ZZ_2 $ \\*
&&&$ [a^2 + a,a^2 + 2a + 3,a^2 + a,-35a^2 + 27a - 39,-121a^2 + 158a - 147] $ & 0 & $ \ZZ_2 $ \\
$ 389a $ & 1064 & $ 6a^2+8a-2 $  & $ [a^2 + 1,4a^2 + 2a + 5,2a^2,26a^2 - 2a - 4,22a^2 - 16a - 16] $ & 0 & $ \ZZ_9 $ \\*
&&&$ [a^2 + 1,4a^2 + 2a + 5,4a^2,54a^2 + 14a - 8,172a^2 - 116a - 176] $ & 0 & $ \ZZ_9 $ \\*
&&&$ [a^2 + 1,4a^2 + 2a + 5,2a^2,-324a^2 + 268a + 396,-2248a^2 - 1372a + 236] $ & 0 & $ \ZZ_3 $ \\*
&&&$ [a^2 + 1,4a^2 + 2a + 5,4a^2,534a^2 - 1276a - 1758,-13880a^2 - 24544a - 16012] $ & 0 & $ 0 $ \\
$ 394a $ & 1080 & $ 6a^2-6a-12 $  & $ [a^2 + a + 1,6a^2 + 2a + 7,6a^2 + 2a - 2,56a^2 - 6a - 8,92a^2 - 58a - 66] $ & 0 & $ \ZZ_7 $ \\*
&&&$ [a^2 + a + 1,6a^2 + 2a + 7,6a^2 + 2a - 2,-184a^2 + 324a - 358,-3268a^2 + 5492a - 3856] $ & 0 & $ 0 $ \\
$ 399a $ & 1097 & $ -5a^2+8a-13 $  & $ [a,3a^2 + a + 2,-1,11a^2 - 2a - 5,4a^2 - 6a - 6] $ & 1 & $ 0 $ \\*
&&&$ [a,3a^2 + a + 2,-1,11a^2 - 2a - 5,4a^2 - 6a - 6] $ & 1 & $ 0 $ \\
$ 405a $ & 1111 & $ 5a^2-a-11 $  & $ [a,2a^2 + 2a,-a - 1,6a^2 - a - 5,-3a - 2] $ & 1 & $ 0 $ \\*
&&&$ [a,2a^2 + 2a,-a - 1,6a^2 - a - 5,-3a - 2] $ & 1 & $ 0 $ \\
$ 405b $ & 1111 & $ 5a^2-a-11 $  & $ [a + 1,3a^2 + a + 3,a^2 + a,-6252a^2 + 10261a - 7932,-336552a^2 + 572642a - 429068] $ & 0 & $ 0 $ \\*
&&&$ [a + 1,3a^2 + a + 3,a^2 + a,13a^2 + a - 2,13a^2 - 3a - 8] $ & 0 & $ \ZZ_5 $ \\*
&&&$ [a + 1,3a^2 + a + 3,a^2 + a,-7a^2 + 6a - 17,-2a^2 + 65a - 54] $ & 0 & $ \ZZ_5 $ \\
$ 406a $ & 1111 & $ a^2-3a-10 $  & $ [a^2 + a,-a^2 + a + 4,-1,-3020a^2 + 8537a - 2126,-20301a^2 + 191896a - 254758] $ & 0 & $ 0 $ \\*
&&&$ [a^2 + a,-a^2 + a + 4,-1,2a + 4,-a^2 + a + 2] $ & 0 & $ \ZZ_5 $ \\*
&&&$ [a^2 + a,-a^2 + a + 4,-1,-5a^2 + 7a - 6,14a^2 + 8a - 42] $ & 0 & $ \ZZ_5 $ \\
$ 421a $ & 1133 & $ 12a^2-7a-7 $  & $ [a^2,4a^2 + 3a + 1,-1,18a^2 - 3a - 15,6a^2 - 14a - 13] $ & 1 & $ \ZZ_3 $ \\*
&&&$ [a^2,4a^2 + 3a + 1,-1,18a^2 - 13a,a^2 + 6a + 9] $ & 1 & $ 0 $ \\
$ 425a $ & 1151 & $ -9a^2+5a-6 $  & $ [a^2,3a^2 + 3a + 1,a^2 - a - 1,14a^2 - 4a - 11,-a^2 - 10a - 7] $ & 1 & $ 0 $ \\*
&&&$ [a^2,3a^2 + 3a + 1,a^2 - a - 1,14a^2 - 4a - 11,-a^2 - 10a - 7] $ & 1 & $ 0 $ \\
$ 426a $ & 1151 & $ -6a^2-5a-5 $  & $ [a,3a^2 + a + 1,a^2 - a - 1,9a^2 - 2a - 6,2a^2 - 5a - 6] $ & 0 & $ \ZZ_3 $ \\*
&&&$ [a,3a^2 + a + 1,a^2 - a - 1,-16a^2 + 48a - 51,-127a^2 + 225a - 199] $ & 0 & $ 0 $ \\
$ 435a $ & 1169 & $ -5a^2+a-8 $  & $ [1,a^2 + a + 4,-a^2 + 1,-119a^2 - 101a - 130,-1672a^2 - 892a - 184] $ & 0 & $ \ZZ_2 $ \\*
&&&$ [1,a^2 + a + 4,-a^2 + 1,-14a^2 + 34a - 20,32a^2 - 51a + 40] $ & 0 & $ \ZZ_4 $ \\*
&&&$ [1,a^2 + a + 4,-a^2 + 1,-19a^2 + 24a - 25,-6a^2 - 78a + 39] $ & 0 & $ \ZZ_2\times\ZZ_2 $ \\*
&&&$ [1,a^2 + a + 4,-a^2 + 1,a^2 - 11a,128a^2 - 312a + 218] $ & 0 & $ \ZZ_2 $ \\
$ 435b $ & 1169 & $ -5a^2+a-8 $  & $ [a,2a^2 + a + 1,a^2,-a^2 + 14a - 7,-11a^2 + 26a - 16] $ & 0 & $ \ZZ_2 $ \\*
&&&$ [a,2a^2 + a + 1,a^2,4a^2 + 84a + 43,-317a^2 + 211a + 298] $ & 0 & $ \ZZ_2 $ \\

\end{longtable}
\end{center}

\iflandscapetable
\end{landscape}
\fi

\end{document}